\documentclass[11pt]{amsart}
\usepackage{stmaryrd}
\usepackage{amssymb}
\usepackage{csvsimple}
\usepackage{siunitx}
\usepackage{hyperref}
\usepackage{url}
\usepackage{graphicx,color}
\usepackage{multirow}
\usepackage{lineno}           
\usepackage[maxbibnames=99,isbn=false,backend=biber,
            url=false,eprint=true,style=alphabetic,
            giveninits=true,labelnumber,defernumbers]{biblatex}

\graphicspath{{pics/},{figs/},{results/}}
\def\R{\mathbb{R}}

\def\cA{\mathcal{A}}

\def\cE{\mathcal{E}}
\def\cF{\mathcal{F}}

\def\cI{\mathcal{I}}

\def\cT{\mathcal{T}}

\def\a{\alpha}

\def\g{\gamma}
\def\d{\delta}

\def\p{\partial}
\def\o{\omega}
\def\veps{\varepsilon}
\def\vrho{\varrho}

\def\O{\Omega}

\def\wto{\rightharpoonup}
\def\transp{{\sf T}}

\def\hx{\widehat{x}}

\newcommand{\dv}[1]{\,{\mathrm d}#1}
\newcommand{\dual}[3][]{#1\langle #2,#3#1\rangle}

\newcommand{\wcheck}[1]{#1\hspace{-.8ex}\mbox{\huge {\lower.45ex \hbox{$\textstyle \check{}$}}} \hspace{.5ex}}

\newcommand{\jump}[1]{\llbracket#1\rrbracket}   
\DeclareMathOperator{\id}{id}

\DeclareMathOperator{\diver}{div}

\DeclareMathOperator{\dist}{dist}
\DeclareMathOperator{\diam}{diam}

\let\oldmarginpar\marginpar
\renewcommand\marginpar[1]{
  \oldmarginpar[\raggedleft\footnotesize #1]
  {\raggedright\footnotesize #1}}

\newtheorem{definition}{Definition}
\newtheorem{lemma}[definition]{Lemma}
\newtheorem{proposition}[definition]{Proposition}
\newtheorem{theorem}[definition]{Theorem}

\newtheorem{remarks}[definition]{Remarks}
\newtheorem{example}[definition]{Example}

\numberwithin{definition}{section}
\definecolor{modmag}{RGB}{179,0,229}

\def\hT{{\widehat{T}}}
\def\V{\mathbb{V}}

\def\hv{\widehat{v}}
\def\hw{\widehat{w}}

\def\tSigma{\widetilde{\Sigma}}
\def\tnabla{\widetilde{\nabla}}
\def\tD{\widetilde{D}}
\def\tH{\widetilde{H}}
\def\tM{\widetilde{M}}
\def\tV{\widetilde{V}}
\def\tW{\widetilde{W}}
\def\tY{\widetilde{Y}}
\def\hY{\widehat{Y}}
\def\tZ{\widetilde{Z}}
\def\tE{\widetilde{E}}
\def\tVV{\widetilde{\V}}

\def\tcA{\widetilde{\cA}}
\def\tcF{\widetilde{\cF}}

\def\tvrho{\widetilde{\vrho}}
\def\ta{\widetilde{a}}
\def\tcE{{\widetilde{\cE}}}
\def\tcT{{\widetilde{\cT}}}
\def\tphi{\widetilde{\phi}}
\def\tcI{\widetilde{\cI}}

\newcommand{\aver}[1]{\{#1\}}
\DeclareMathOperator{\Diver}{Div}

\newcommand{\D}{\nabla}
\def\e{\varepsilon}
\def\AA{\mathcal{A}}
\def\weak{\rightharpoonup}

\def\d{\partial}

\begin{document}
\title[Thin sheet folding]{Modeling and simulation of thin sheet folding}
\author[S. Bartels]{S\"oren Bartels}
\address{Abteilung f\"ur Angewandte Mathematik,  
Albert-Ludwigs-Universit\"at Freiburg, Hermann-Herder-Str.~10, 
79104 Freiburg i.~Br., Germany}
\email{bartels@mathematik.uni-freiburg.de}

\author[A. Bonito]{Andrea Bonito}
\address{Texas A\& M University, College Station, TX 77843, USA}
\email{bonito@tamu.edu}

\author[P. Hornung]{Peter Hornung}
\address{Fakult\"at Mathematik, Technische Universit\"at Dresden,
Zellescher Weg 12--14, 01069 Dresden, Germany}
\email{peter.hornung@tu-dresden.de}

\date{\today}
\renewcommand{\subjclassname}{
\textup{2010} Mathematics Subject Classification}
\begin{abstract}
The article addresses the mathematical modeling of the folding of a thin elastic
sheet along a prescribed curved arc. A rigorous model reduction from a general
hyperelastic material description is carried out under appropriate scaling 
conditions on the energy and the geometric properties of the folding arc in 
dependence on the small sheet thickness. The resulting two-dimensional model is 
a piecewise nonlinear Kirchhoff plate bending model with a continuity condition
at the folding arc. A discontinuous Galerkin method and an iterative scheme are
devised for the accurate numerical approximation of large deformations.
\end{abstract}


\maketitle

\section{Introduction}
Because of their relevance in the development of new technologies
bending theories for thin sheets have attracted considerable
attention within applied mathematics in the past decades,
with renewed activity following the seminal article~\cite{FrJaMu02}. In
the present article the folding of thin elastic sheets along a
prepared curved arc is considered which naturally leads to bending effects,
cf.~Figure~\ref{fig:model_exp}. This setting has only been partially
addressed mathematically, e.g.,~\cite{DunDun82,BaBoHo22-prep}, but has recently attracted
considerable attention in applied sciences, 
cf., e.g.~\cite{Specketal15,CheMah19,ChDuMa19,PeHaLa19-book,CGGP19,LiPlJaFe21} and references therein. 
Particular applications
arise in the design of cardboxes and bistable switching devices that
make use of corresponding flapping mechanisms.
It is our aim to derive a mathematical description via a rigorous dimension
reduction from three-dimensional hyperelasticity and to devise
effective numerical methods that correctly predict large
deformations in practically relevant settings.  

To describe our approach we let $S \subset \R^2$ be a bounded
Lipschitz domain that represents the midplane of an asymptotically
thin sheet and let $\Sigma \subset S$ be a curve (with endpoints on $\d S$)
that models the crease, i.e., 
the arc along which the sheet is folded. The corresponding three-dimensional
model involves a thickness parameter $h>0$ and the thin domain
$\O_h = S \times (-h/2,h/2)$. The material is weakened (or damaged)
in a neighbourhood of width $r>0$ around the arc $\Sigma$ and 
we consider for given functions $W$ and $f_{\veps,r}$ the
hyperelastic energy functional for a deformation $z:\O_h \to \R^3$
\[
E^h(z) = \int_{\O_h} f_{\veps,r}(x) W(\nabla z) \dv{x}.
\]
Here $f_{\veps,r}(x) \in (0,1]$ is small with value $\veps>0$ close to 
the arc $\Sigma$ and approximately $1$ away from the arc, the function $W$ is
a typical free energy density. Hence, the factor $f_{\veps,r}$ models a
reduced elastic response of the material close to $\Sigma$. 
By appropriately relating the
thickness $h$, the intactness fraction $\veps$, and width
$r$ of the prepared region, we obtain for $(h,\veps,r) \to 0$
a meaningful dimensionally reduced model which seeks a minimizing
deformation $y:S\to \R^3$ for the functional 
\[
E_K(y) = \frac{1}{24} \int_{S \setminus \Sigma} Q(A) \dv{x'},
\]
where $A$ is the second fundamental form related to the 
parametrization $y$ which is weakly differentiable in $S$ with
second weak derivatives away from $\Sigma$, i.e., we consider
\[
y\in W^{2,2}(S\setminus \Sigma;\R^3) \cap W^{1,\infty}(S;\R^3).
\]
The quadratic form $Q$ is obtained by a pointwise relaxation procedure of the
Hessian of $W$. For isotropic materials it can explicitly be represented
in terms of the Lam\'e coefficients, cf.~\cite{FrJaMu02}.
Furthermore, the deformation $y$ is required to satisfy the pointwise isometry condition
\[
(\nabla y)^\transp (\nabla y) = I,
\]
which implies that no shearing and stretching effects occur.
The minimization of $E_K$ is supplemented by boundary conditions
and possible body forces. Note that for our scaling of the parameters
$\veps$, $r$, and $h$, e.g., $\veps = o(h)$ and $r = O(h)$,
no energy contributions such as a penalization of the folding angle
arise from the crease,
i.e., the material can freely fold along the arc and $y$ is in general
only continuous on $\Sigma$. In our analysis, the damaged set is required to 
be wide enough to ensure that the strain of the folded sheet can easily remain 
of order one, and the damaged material should be soft enough to ensure
that the energetic contribution of the fold becomes asymptotically negligible. On the other hand,
if the damaged material is too soft, the sheet could fall apart;
our compactness ensures the continuity of the asymptotic 
deformation across the fold.

In the absence of a crease $\Sigma$ the model coincides with
Kirchhoff's plate bending functional describing, e.g., the deformation of paper,
which was rigorously derived in~\cite{FrJaMu02}. We slightly modify the arguments
given there to take into account the fact that any asymptotic deformation which does
not only bend but
also folds has infinite Kirchhoff energy. This is because
the Hessian corresponding to a folding deformation is not square integrable.
Other variants of the setting from \cite{FrJaMu02} have been addressed,
in \cite{FJMM03,FrJaMu02b,HorVel18,Velc15,Schm07} and many others. 

\begin{figure}[h]
\begin{center}
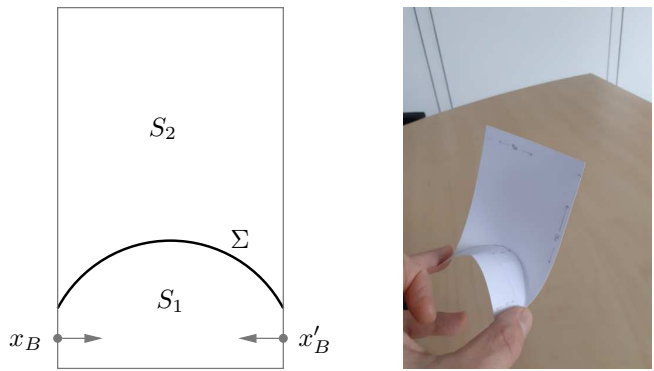 
\end{center}
\caption{\label{fig:model_exp} {\em Left:} Geometry of a prepared
elastic sheet $S$ with folding arc $\Sigma$ that separates regions
$S_1$ andf $S_2$; boundary conditions are imposed in the points
$x_B,x_B'$. {\em Right:} When the boundary points $x_B$ and $x_B'$
are moved towards each other a flapping mechanism occurs.}
\end{figure}
\medskip

A typical setting and an experiment are shown in Figure~\ref{fig:model_exp}
and a coarse visual inspection indicates that the scaling relation $r \sim h$
together with a sufficiently strong compression of the material along the
arc lead to the desired folding effect. 
Interesting phenomena take place at the folding arc. It was
observed in~\cite{DunDun82} that a deformation on one side of
the arc locally restricts possible deformations on the other
side. In fact, only a finite number of scenarios is possible and
either the gradients coincide along the arc resulting in a smooth deformation
or a discontinuity occurs and the deformation is locally up to a sign uniquely determined. 
This important effect is a result of the isometry condition and the related physical
property that thin elastic sheets are unshearable in the
bending regime. Moreover, this effect arises in biology and inspires the
development of new technologies and design in architecture~\cite{Specketal15}.

Our numerical method to approximate minimizers is based on a discontinuous
Galerkin method from~\cite{BoNoNt21} which generalizes the approach based on a
nonconforming method of~\cite{Bart13a}. It allows us
to define a discrete curve $\tSigma$ as a union of element sides
that approximates $\Sigma$ and to account for possible discontinuities
of deformation gradients along $\Sigma$ by simply removing typical jump
terms in the discrete formulation. The continiuity condition on the deformation~$y$
along the interface $\Sigma$ is similar to a simple support boundary condition
in linear bending theories. For such problems the plate paradox, cf.~\cite{BabPit90},
states that convergence of approximations resulting from polyhedral
approximations of curved domains or interfaces fails in general. We therefore
consider piecewise quadratic approximations $\tSigma$ of $\Sigma$.
We note however that we do not observe significant differences to
approximations obtained with piecewise linear arcs $\tSigma$ in the nonlinear setting
under consideration.

We assume for simplicity and without loss of generality 
that $Q(A) = |A|^2$ and use that the
Frobenius norm of the second fundamental form equals the
Frobenius norm of the Hessian in the case of an isometry~$y$, i.e.,
$|A|^2 = |\nabla^2y|^2$. Our numerical method then uses discontinuous
deformations $\tY: S \to \R^3$ from an isoparametric
finite element space $\tVV^3 \subset L^2(S;\R^3)$ subordinated
to a partitioning $\tcT$ of $S$ and a suitably defined
discrete Hessian $\tH(\tY)$ in $S\setminus \tSigma$ which
define the discrete energy functional
\[\begin{split}
\tE_K(\tY) =&  \frac{1}{24} \int_{S\setminus \Sigma} |\tH(\tY)|^2 \dv{x'} \\
& + \frac{\g_0}{2} \int_{\cup \tcE^a} h_{\tcE}^{-3} |\jump{\tY}|^2 \dv{s}
+ \frac{\g_1}{2} \int_{\cup \tcE^a \setminus \tSigma} h_{\tcE}^{-1} |\jump{\tnabla \tY}|^2 \dv{s}.
\end{split}\]
The last two terms are typical stabilization terms with the mesh-size
function $h_\tcE$ on the skeleton $\tcE$ of the triangulation $\tcT$
that guarantee coercivity and enforce continuity
of $\tY$ and the elementwise gradient $\tnabla y$ across 
interelement sides or essential boundary conditions on certain boundary
sides as the mesh-size tends to zero. The union of all such sides is the set
of active sides which is denoted by $\tcE^a$. Crucial here is that
the penalized continuity of $\tnabla \tY$ is not imposed along the
discrete folding arc $\tSigma$ and that related consistency terms do not
enter in the definition of the discrete Hessian $\tH$.
The important isometry condition is imposed up to a tolerance via 
a sum of integrals on elements, i.e., we require that
\begin{equation}\label{e:discrete_admissible_set}
\tY \in \tcA = \Big\{\tZ \in \tVV^3: 
\sum_{T\in \tcT} \Big|\int_T (\tnabla \tZ)^\transp (\tnabla \tZ) - I \dv{x'} \Big| 
\le \tvrho  \Big\}
\end{equation}
for a suitable tolerance $\tvrho>0$. To obtain accurate approximations
of large deformations choosing a small parameter is desirable. Its
particular choice is dictated by available density results. If, e.g.,
the density of smooth folding isometries can be guaranteed 
similar to~\cite{Horn11b} even a pointwise controlled violation criterion
can be used;
otherwise a condition $\tvrho \ge c' h_\tcT$ has to be satisfied, cf.~\cite{BoNoNt21}.
Boundary conditions on a part of the boundary $\p S$
are included in the discrete problem via an appropriate definition of 
the jump terms. A justification of the discrete energy functionals
via $\Gamma$ convergence as in~\cite{Bart13a,BaBoNo17,BoNoNt21} is 
in preparation. 

The iterative solution of the constrained minimization problem 
follows the ideas of~\cite{Bart13a,BaBoNo17,BoNoNt21} and is realized
by a discrete gradient flow with a suitable linearization of the constraint.
In particular, we consider the linear space $\tcF[\tZ]$ of variations for a 
given deformation $\tZ\in \tcA$ defined via
\[
\tcF[\tZ] = \Big\{\tW \in \tVV^3: \int_T (\tnabla \tZ)^\transp (\tnabla \tW)
+ (\tnabla \tW)^\transp (\tnabla \tZ) \dv{x'}  = 0 \, \text{ for all } T\in \tcT \Big\}. 
\]
We furthermore let $(\cdot,\cdot)_{\star}$ be an inner product and $\tau>0$ 
a step size. Given an approximation $\tY^{k-1} \in \tcA$ we look for a 
correction $d_t \tY^k \in \tcF[\tY^{k-1}]$ such that
\begin{equation}\label{e:gradient_flow}
\big(d_t \tY^k, \tV\big)_{\star} + \ta_K\big(\tY^{k-1} + \tau d_t \tY^k, \tV\big) = 0
\end{equation}
for all $\tV \in \tcF[\tY^{k-1}]$ and with the discrete bilinear form $\ta_K$
associated with the discrete quadratic energy functional $\tE_K$.
The existence of a unique solution $d_t \tY^k$ is an immediate consequence of
the Lax--Milgram lemma and we define the new approximation 
\[
\tY^k = \tY^{k-1} + \tau d_t \tY^k.
\]
This implies the interpretation of the symbol $d_t$ as a backward difference
quotient. By choosing $\tV = d_t \tY^k$ one directly obtains the energy decay 
property for $\ell =0,1,\dots$, 
\[
\tE_K[\tY^\ell] + \tau \sum_{k=1}^\ell \|d_t\tY^k\|_{\star}^2  \le \tE_K[\tY^0] = \tE_K^0
\]
in particular, we see that $\|d_t \tY^k\|_{\star} \to 0$ as $k\to \infty$,
i.e., the iteration becomes stationary. Because of the orthogonality relation
included in the space $\tcF[\tY^{k-1}]$ we can bound the violation of the
isometry constraint by repeatedly replacing $\tY^k = \tY^{k-1} + \tau d_t \tY^k$, i.e.,
for all $T\in \tcT$ we have 
\[\begin{split}
\int_T [\tnabla \tY^\ell]^\transp \tnabla \tY^\ell - I \dv{x'}
&= \int_T [\tnabla \tY^0]^\transp \tnabla \tY^0 - I 
+ \tau^2 \sum_{k=1}^\ell [\tnabla d_t \tY^k]^\transp \tnabla d_t \tY^k \dv{x'}.
\end{split}\]
If the discrete gradient flow metric $(\cdot,\cdot)_{\star}$ 
controls the $L^2$ norm of the elementwise gradient,
we obtain from the energy decay law the estimate
\[
\sum_{T\in \tcT} \big| \int_T [\tnabla \tY^\ell]^\transp \tnabla \tY^\ell - I  \dv{x'} \big|
\le \veps^0 + c_\star \tE_K^0 \tau,
\]
where $\veps^0$ is the initial isometry violation. 
In particular, we see that if $\tau$ and $\veps^0$ are sufficiently small, an arbitrary
accuracy can be achieved and we have $\tY^\ell\in \tcA$ independently of the
number of iterations $\ell \ge 0$. 

Using the numerical scheme we simulate various scenarios that are
motivated by practical applications, e.g., how the shape of the folding arc 
affects the flapping mechanism, or address subtle analytical features
of solutions, e.g., the occurrence of energy concentrations  
when the curve $\Sigma$ has a kink. Besides that we illustrate the robustness
of the numerical method with respect to the choice of stabilization parameters and
discuss the construction of
suitable deformations that serve as starting values in the discrete gradient
flow. Our experiments show that large deformations in highly nontrivial settings can
be accurately computed with moderate resolution. 

The outline of the article is as follows. In Section~\ref{sec:prelim}
we describe the general setup and the dimensionally reduced model.
Its rigorous derivation is given in Section~\ref{sec:dim_red}. The
discontinuous Galerkin finite element method is derived and stated in
Section~\ref{sec:fem}.
Numerical experiments are reported in Section~\ref{sec:num_exp}.


\section{Preliminaries}\label{sec:prelim}

\subsection{Hyperelasticity for plates}

For a bounded Lipschitz domain $S\subset\R^2$ we consider a plate of thickness $h>0$
occupying the domain $\Omega_h = S\times I_h$ in the reference configuration. The elastic
energy stored in the configuration determined by a deformation $z : \Omega_h\to\R^3$ is given
by
$$
\int_{\Omega_h} W(\D z) \dv{x}.
$$
Here $W$ is a frame indifferent stored energy function and as in~\cite{FrJaMu02}
we impose the following conditions:

\begin{enumerate}
\item[(H1)] $W\in C^0(\R^{3\times 3})$ and $W\in C^2$ in a neighbourhood of $SO(3)$.
\item[(H2)] $W$ is frame indifferent, i.e., $W(F) = W(RF)$ for all $F\in\R^{3\times 3}$
and all $R\in SO(3)$. Moreover, $W(I) = 0$.
\item[(H3)] There is a constant $C$ such that $\dist^2_{SO(3)}\leq CW$;
here $\dist_{SO(3)} : \R^{3\times 3}\to [0, \infty)$
denotes the distance function from the set $SO(3)$. \label{coercive}
\end{enumerate}

To analyze the limiting behaviour as $h\to 0$ it is convenient to work on the fixed
domain 
$$
\Omega = S\times I,
$$
where $I = (-1/2,1/2)$. We define a rescaled deformation $y^h : \Omega\to\R^3$ by
setting $y^h(x', x_3) = z(x', hx_3)$. Then the (re-scaled) elastic energy is given by
$$
\widetilde{E}^h(y^h) = \int_{\Omega} W(\D_h y^h) \dv{x},
$$
where $\D_h y^h = (\D' y^h\ |\ \frac{1}{h}\d_3 y^h)$ and $\D' = (\d_1, \d_2)$.

\subsection{Notation}
Throughout this article we use  standard notation related to Sobolev spaces,
e.g., $W^{s,p}(U;\R^\ell)$ denotes the set of $s$ times weakly differentiable,
$\R^\ell$-valued functions in $L^p$ whose weak derivatives are $p$-integrable.
$L^p$ norms are often used without specifying a domain when there is no ambiguity, 
and we abbreviate the $L^2$ norm on $S$ by $\|\cdot\|$.
We occasionally omit target domains $\R^\ell$ when this is clear from the
context. The open ball of radius $r > 0$ around a point $x\in\R^n$ is denoted by $B_r(x)$.
For integral functionals ocurring below, it is often useful to specify their
integration domains explicitly, e.g., we write
\[
E(y; \tilde{S}) = \int_{\tilde{S}} F(y) \dv{x'}.
\]
For the canonical choice, e.g., $\tilde{S} = S$, this argument is usually
omitted. For $S\subset\R^2$ we identify maps defined on $S$ with their
trivial extension to $S\times\R$.

%
%
%

\subsection{Folded plates}

Our aim is to modify the arguments from \cite{FrJaMu02} in order to
allow for folding effects along a prescribed curve, see Figure~\ref{fig:model_exp}.
In applications, the folding curve $\Sigma$ is prescribed by weakening
the material along it.
\\
Throughout this article $S\subset\R^2$ is a bounded Lipschitz domain.
From now on $\Sigma\subset S$ is a Jordan arc with both endpoints on the same
connected component of $\d S$. More precisely, let $\sigma : [0, 1]\to\R^2$
be continuous and injective, set $\Sigma = \sigma(0, 1)$
and $\d\Sigma = \{\sigma(0), \sigma(1)\}$ and assume that $\Sigma\subset S$
and that $\d\Sigma$ is contained in one single connected
component of $\d S$. Then $S\setminus\Sigma$ consists of precisely two
connected components $S_1$ and $S_2$. We assume, in addition, that $\Sigma$
is such that both $S_1$ and $S_2$ are Lipschitz domains.
\\
This latter hypothesis
entails a great deal of regularity on $\Sigma$. In particular, $\Sigma$ is locally
a Lipschitz graph. Therefore, the area of the sets
$$
\Sigma_R = B_R(\Sigma) = \bigcup_{x\in \Sigma} B_R(x)
$$
converges to zero as $R\downarrow 0$.
As explained in the introduction, we let $r_h$, $\e_h >0$ be
parameters that define the width of the prepared region and the amount
of the material intactness. We then define $f^h : S\to [0, \infty]$ by
\begin{equation}
\label{defB}
f^h = \e_h \chi_{\Sigma_{r_h}} + 1 -  \chi_{\Sigma_{r_h}},
\end{equation}
where $\chi_M$ denotes the characteristic function of a set $M$.
With this we consider the (re-scaled) three-dimensional energy functional
$E^h : W^{1,2}(\Omega;\R^3)\to [0, \infty]$ 
\begin{equation}
\label{defE}
E^h(y) = \int_{\Omega} f^h(x') W\big(\D_h y^h(x)\big) \dv{x}.
\end{equation}
Passing to the thin film limit $h\to 0$ leads to a pointwise isometry
constraint, but does not exclude and discontinuities of the
gradients across the arc $\Sigma$:
we are led to the set of asymptotically admissible deformations
\[\begin{split}
\AA(S, \Sigma) &= \{u\in W^{1,2}(S;\R^3)\cap W^{2,2}(S\setminus\Sigma;\R^3) : \\
& \qquad \qquad (\D u)^\transp(\D u) = I\mbox{ a.e. on }S\}.
\end{split}\]
The corresponding asymptotic energy functional $E_K : W^{1,2}(S;\R^3)\to [0, \infty]$
is defined as 
\begin{equation}
\label{defEK}
E_K(y) =
\begin{cases}
\frac{1}{24}\int_{S\setminus \Sigma} Q(A) \dv{x'} &\mbox{ if }u\in\AA(S, \Sigma),  \\
+\infty &\mbox{ otherwise.}
\end{cases}
\end{equation}
Here, $A$ is the second fundamental form of the surface parametrized
by $y$ with unit normal $n = \p_1 y \times \p_2 y$, i.e.,
\[
A =  (\nabla n)^\transp (\nabla y),
\]
and $Q$ is obtained by relaxing, over the third column and row, the
quadratic form corresponding to the Hessian $D^2 W(I)$ of $W$
at the identity matrix, i.e.,
\[
Q(A) = \min_{d\in \R^3} D^2W (I)[(A\,|\,d),(A\,|\,d)],
\]
where for given $A\in \R^{2\times 2}$ the matrix $(A\,|\,d)\in \R^{3\times 3}$ is
obtained by consistently appending a row and column defined by $d\in \R^3$. Note
that by hypotheses~(H1)-(H3) we have the Taylor expansion
\[
W(I+hF) = \frac{h^2}2  D^2W(I)[F,F] + o(h^2)
\]
for $F\in \R^{3\times 3}$ and $h>0$. 

\begin{remarks}
(i) Observe that every $y\in\AA(S, \Sigma)$ belongs to
$W^{1,\infty}(S;\R^3)$, because $y\in W^{1,2}(S;\R^3)$ and its
derivatives are bounded almost everywhere on $S$ since
$(\D y)^\transp(\D y) = I$ almost everywhere. In particular, $y$ is continuous
on $S$.
\\
(ii) We recall that a Lipschitz function $f$ is in $W^{2,2}(S\setminus\Sigma)$
precisely if there is an $F\in L^2(S;\R^{2\times 2})$ such that
$
\D^2 f = F
$
in the sense of distributions on $S\setminus\Sigma$.
\end{remarks}

\section{Gamma-Convergence}\label{sec:dim_red}

The purpose of this section is to prove the following result.

\begin{theorem}\label{gamma}
%
Let $\e_h$, $r_h\in (0, \infty)$
be null sequences satisfying
\begin{equation}
\label{parameter-1}
\limsup_{h\to 0} \frac{h^2}{\e_h} < \infty
\end{equation}
and
\begin{equation}
\label{parameter-2}
\limsup_{h\to 0} \frac{h}{r_h} < \infty
\end{equation}
as well as
\begin{equation}
\label{parameter-3}
\limsup_{h\to 0}\frac{\e_h r_h}{h^2} = 0.
\end{equation}
Define $f^h$ as in \eqref{defB}, define
$E^h$ as in \eqref{defE} and define $E_K$ as in \eqref{defEK}.
\\
Then deformations with finite bending energy
are compact and $\frac{1}{h^2}E^h$
Gamma-converges to $E_K$. More precisely:
\begin{enumerate}
\item \label{gamma:compact}
Assume that $y^h\in W^{1,2}(\Omega, \R^3)$ are such that
$$
\limsup_{h\to 0}\frac{1}{h^2}E^h(y^h) < \infty.
$$
Then there exists
a subsequence (not relabelled) and $y\in\AA(S, \Sigma)$ such that
$y^h\weak y$ weakly in $W^{1,2}(\Omega)$ and
locally strongly in $W^{1, 2}((S\setminus\Sigma)\times I)$.
\item Assume that $y^h\weak y$ weakly in $W^{1,2}(\Omega)$. Then
$$
E_K(y)\leq\liminf_{h\to 0}E^h(y^h).
$$
\item \label{gamma:upper}
Let $y\in W^{1,2}(\Omega)$. Then there exist $y^h\in W^{1,2}(\Omega)$
such that
$$
\lim_{h\to 0}
\frac{1}{h^2}E^h(y^h) = E_K(y).
$$
\end{enumerate}
\end{theorem}

\begin{remarks}
(i) The bound \eqref{parameter-1} ensures that the damaged material
is not too soft. This is used in the proof of the
compactness result, Part \ref{gamma:compact} of Theorem \ref{gamma},
as it rules out discontinuities of the asymptotic deformation across $\Sigma$.
\\
The bound \eqref{parameter-2} requires the damaged part of the material
to be wide enough with respect to the thickness of the sheet, 
while \eqref{parameter-3} asserts that the damaged portion of the material
should be soft enough to ensure that the fold does not contribute to the asymptotic
energy. Conditions \eqref{parameter-2} and \eqref{parameter-3}
are used in Part \ref{gamma:upper} of Theorem \ref{gamma}, as they
exclude excessive strain. \\
(ii) Observe that \eqref{parameter-1} through \eqref{parameter-3} are met, for instance,
if $\e_h\sim h^2$ and $r_h\sim h$.
\end{remarks}

Theorem \ref{gamma} is a consequence of Proposition \ref{lower&compact}
and of Proposition \ref{proposition:upper} below. Both of them
rely on arguments and results in \cite{FrJaMu02}.

\subsection{Compactness and lower bound}

\begin{proposition}\label{lower&compact}
Let $\e_h$, $r_h\downarrow 0$ as $h\downarrow 0$ and assume 
that \eqref{parameter-1} is satisfied.
If $y^h\in W^{1,2}(S, \R^3)$ satisfy
$$
\limsup_{h\to 0}
\frac{1}{h^2}E^h(y^h) < \infty,
$$
then there exists a map $y\in \AA(S, \Sigma)$
such that, after taking subsequences, $y^h\weak y$ weakly in
$W^{1,2}(\Omega, \R^3)$ and
locally strongly in $W^{1, 2}((S\setminus\Sigma)\times I, \R^3)$
as $h\downarrow 0$.
Moreover,
$$
E_K(y; S_i) \leq\liminf_{h\to 0}\frac{1}{h^2} E^h(y^h; S_i\times I)
\mbox{ for }i = 1, 2.
$$
\end{proposition}
\begin{proof}
We omit the index $h$ in $\e_h$ and $r_h$;
the letter $C$ denotes constants that do not depend on $h$.
By the definition of $f^h$ in \eqref{defB} we have
$$
\int_{\Omega\setminus (\Sigma_r\times I)}W(\D_h y^h) \dv{x} \leq Ch^2\leq C.
$$
On the other hand, by~\eqref{parameter-1},
\begin{align*}
\e\int_{\Sigma_r\times I} W(\D_h y^h) \dv{x} =
\int_{\Sigma_r\times I} f^h W(\D_h y^h) \dv{x} \leq Ch^2 \leq C\e. 
\end{align*}
From hypothesis (H3) on $W$ we deduce that
$$
\int_{\Omega} \dist^2_{SO(3)}(\D_h y^h) \dv{x} \leq C \int_{\Omega} W(\D_h y^h) \dv{x} \leq C.
$$
Hence $(\D_hy^h)$ is uniformly bounded in $L^2(\Omega)$ due to
the hypotheses on $W$.
This implies that there exists $y\in W^{1,2}(S)$ such that $y^h\weak y$ weakly
in $W^{1,2}(\Omega)$, after taking subsequences. Indeed,
we first notice that $(\D y^h)$ is uniformly bounded in $L^2(\Omega)$
and therefore, after taking a subsequence (not relabelled),
we see that there is some $y\in W^{1,2}(\Omega)$ such that
$y^h\weak y$ weakly in $W^{1,2}(\Omega)$. Then we note that
$\|\d_3 y^h\|_{L^2(\Omega)}\leq Ch\to 0$ implies that $\d_3 y = 0$.
Hence $y$ does not depend on $x_3$ and therefore we can
identify it with a map (denoted by the same symbol) in $W^{1,2}(S)$.
\\
Since $\int_{(S\setminus\Sigma_r)\times I)} W(\D_h y^h) \dv{x} \leq Ch^2$
and $r\to 0$, by monotonicity of the integral we conclude that
for any ($h$-independent) $R > 0$ we have
$$
\int_{(S\setminus\Sigma_R)\times I} W(\D_h y^h) \dv{x} \leq Ch^2
$$
for all small enough $h$. The constant $C$ does not depend on $R$.
\\
Now fix a small $R > 0$ and define $S_i^R = S_i\setminus\Sigma_R$ for $i = 1$, $2$;
both are Lipschitz domains.
We can apply \cite[Theorems 4.1 and 6.1 (i)]{FrJaMu02} on each $S_i^R$.
Hence $y^h\to y$ strongly in $W^{1,2}(S_i^R\times I)$
and $y\in\AA(S, \Sigma_R)$ with
\begin{equation}
\label{lowerb-1}
\frac{1}{24}\int_{S_i^R} Q(A) \dv{x'} \leq\liminf_{h\to 0}h^{-2}\int_{S_i^R\times I}
W(\D_h y^h) \dv{x} \leq C.
\end{equation}
Here $A$ is the second fundamental form of $y$ on 
$S\setminus\Sigma_R = S_1^R\cup S_2^R$.
Since $y$ is an isometric immersion, by \cite[Proposition 6]{FJM2} we have
$$
|\D^2 y| = |A|\mbox{ almost everywhere on }S\setminus\Sigma_R.
$$
Hence \eqref{lowerb-1} implies that
\begin{equation}
\label{lowerb-1-monthslater}
\|\D^2 y\|_{L^2(S_1^R)} + \|\D^2 y\|_{L^2(S_2^R)}\leq C.
\end{equation}
This is true for all $R$ and the constant $C$ does not depend on $R$.
As noted earlier, the area of $\Sigma_R$ converges to $0$ as $R\to 0$.
Hence $y\in W^{2,2}(S\setminus\Sigma)$
and
$$
\int_{S_i} Q(A) \dv{x'} = \limsup_{R\to 0}\int_{S_i^R} Q(A) \dv{x'} < \infty.
$$
Summarising, we have $y^h\to y$
locally strongly in $W^{1,2}((S\setminus\Sigma)\times I)$
and $y\in\AA(S, \Sigma)$.
\\
According to \eqref{lowerb-1}, for every small $R > 0$ we have
\begin{align*}
E_K(y; S_i^R)&\leq\liminf_{h\to 0}h^{-2}E^h(y^h; S_i^R\times I)
\\
&\leq\liminf_{h\to 0}h^{-2}E^h(y^h; S_i\times I).
\end{align*}
The right-hand side does not depend on $R$. Taking the supremum
over all small $R > 0$, we therefore see that
$$
E_K(y; S_i) \leq \liminf_{h\to 0}h^{-2}E^h(y^h; S_i\times I).
$$
\end{proof}

%

\subsection{Recovery sequence}

In the proof of Proposition \ref{proposition:upper} below
we will use the following lemma.

\begin{lemma}
\label{uplem}
Let $U\subset\R^2$ be a bounded Lipschitz domain. Then there exists a constant
$\delta > 0$,
depending only on the Lipschitz constant of $U$, such that
the following is true: if $M\subset U$ satisfies $|M| < \delta(\diam U)^2$
and if we set
$$
R = \sqrt{\frac{2|M|}{\delta}},
$$
then $B_R(x)$ intersects $U\setminus M$ for each $x\in U$.
\end{lemma}
\begin{proof}
This is standard; we include the proof 
for convenience.
By Definition 1.3 in \cite[Chapter III]{Giaquinta} and the remark 
following it, there exists a constant $\delta > 0$, depending only
on the Lipschitz constant of $U$, such that
$|B_{\rho}(z)\cap U|\geq \delta\rho^2$ whenever $z\in U$ and $\rho < \diam U$.
\\
If $B_R(x)$ did not intersect
$U\setminus M$, then $B_R(x)\cap U\subset M$ and thus we would have
$$
2|M| = \delta R^2 \leq |B_R(x)\cap U| \leq |M|,
$$
a contradiction.
\end{proof}

The main result of this section is the following proposition.

\begin{proposition}\label{proposition:upper}
Let $\e_h$, $r_h\in (0, \infty)$ be null sequences
satisfying \eqref{parameter-2} and \eqref{parameter-3}
and let $y\in\AA(S, \Sigma)$.
Then there exist $y^h\in W^{1,2}(\Omega, \R^3)$ such that
$y^h\weak y$ weakly in $W^{1,2}(\Omega, \R^3)$ and
\begin{equation}
\label{gammalimsup-energies}
\lim_{h\downarrow 0}E^h(y^h) = E_K(y).
\end{equation}
\end{proposition}
\begin{proof}
As before $S_{1,2}$ denote the connected components of $S\setminus\Sigma$.
Denote the restriction of $y$ to $S_i$ by $u_i$ and
denote by $n_i$ the normal to $u_i$.
In this proof we will write $\e$ instead of $\e_h$
and $r$ instead of $r_h$. The letter $C$ denotes constants that do
not depend on $h$ as $h\downarrow 0$.
\\
Let $U\subset\R^2$ be an open ball containing the closure of $S$.
As each $S_i$ is a Lipschitz domain,
by \cite{Stein} we can extend each $u_i$ and each $n_i$ to maps
\begin{equation}
\label{extended}
\begin{split}
u_i &\in W^{2,2}(\R^2, \R^3)\cap W^{1,\infty}(\R^2, \R^3)
\\
n_i &\in W^{1,2}(\R^2, \R^3)\cap L^{\infty}(\R^2, \R^3)
\end{split}
\end{equation}
supported in $U$.
Notice that we use the same symbols to denote the extended maps as for the
original ones. The norms of the extended maps
can be bounded by those of the original ones, up to a factor that only
depends on $U$, $S$ and on $\Sigma$.
\\
As in the proof of \cite[Theorem 6.1 (ii)]{FrJaMu02}
we truncate the maps $n_i$ and $u_i$
and thus obtain sequences of maps $n_i^h$ and $u_i^h$ satisfying the bound
\begin{equation}
\label{uppera-1}
\|(\D')^2 u_i^h\|_{L^{\infty}(U)} + \|\D' n_i^h\|_{L^{\infty}(U)}\leq\frac{1}{h},
\end{equation}
while at the same time there is a set $M^h\subset S$ with
\begin{equation}
\label{mhbd}
\limsup_{h\to 0} \frac{|M^h|}{h^2} = 0
\end{equation}
such that
\begin{equation}
\label{notmh}
u_i^h = u_i\mbox{ and }n_i^h = n_i\mbox{ on }S\setminus M^h.
\end{equation}
%
%
Since $|M^h|\to 0$ by \eqref{mhbd}, Lemma \ref{uplem} 
shows that there is a constant $\delta$ depending only
on $S$ and $\Sigma$, such that choosing
$$
\rho_h = \sqrt{\frac{2|M^h|}{\delta}}
$$
we have
\begin{equation}
\label{brho}
B_{\rho_h}(x_0)\cap S\setminus M_h\neq\emptyset\mbox{ for all }x_0\in S.
\end{equation}
By \eqref{mhbd} we have
\begin{equation}
\label{rhobd}
\limsup_{h\to 0}
\frac{\rho_h}{h} = 0.
\end{equation}
We claim that, for $i = 1, 2$, the following $L^{\infty}$ bounds are satisfied
for a constant $C$ depending only on $y$, $S$ and $\Sigma$:
\begin{equation}
\label{uppera-2}
\frac{1}{\rho_h}|u_i^h - u_i| + |\D' u_i^h| + |n_i^h|\leq C \mbox{ almost everywhere on }S.
\end{equation}
In fact, \eqref{extended} implies that
\begin{equation}
\label{extended-1}
|\D' u_i| + |n_i|\leq C \mbox{ almost everywhere on }\R^2.
\end{equation}
Since $u_i^h = u_i$ (hence $\D' u_i^h = \D' u_i$) 
and $n_i^h = n_i$ almost everywhere on $S\setminus M^h$ by \eqref{notmh}, we clearly have
\begin{equation}
\label{uppera-78}
|\D' u^h_i| + |n^h_i|\leq C \mbox{ almost everywhere on }S\setminus M^h.
\end{equation}
On the other hand, \eqref{uppera-1} shows that the Lipschitz
constants of $\D' u_i^h$ and of $n_i^h$ on $U$ are bounded by $1/h$.
By \eqref{brho}, for all $x\in S$ there is a $y\in S\setminus M^h$ such that
$$
|n_i^h(x)| \leq |n_i^h(y)| + |n_i^h(x) - n_i^h(y)|\leq C + \frac{1}{h}|x - y|
\leq C + \frac{\rho_h}{h}.
$$
Here we used \eqref{uppera-78} to estimate $|n_i^h(y)|$.
In view of \eqref{rhobd} this implies the bound on
$|n_i^h|$ asserted in \eqref{uppera-2}. The bound on $|\D' u_i^h|$ is proven
similarly.
\\
In particular, the Lipschitz constants of $u_i^h : S\to\R^3$ are uniformly bounded.
Since $u_i$ is Lipschitz as well, the maps
$u_i^h - u_i$ are Lipschitz on $S$ with uniformly bounded Lipschitz constants 
as $h\downarrow 0$. Since $u_i^h - u_i = 0$ on $S\setminus M^h$,
we deduce from \eqref{brho} that $|u_i^h - u_i|\leq C\rho_h$ on $S$. This concludes the proof of \eqref{uppera-2}.
\\
We will now define the recovery sequence. In order to do so, 
for each $h$ let $\eta^h\in C^{\infty}(S, [0, 1])$
be a smooth cutoff function with $\eta^h = 1$ on
$S_1\setminus\Sigma_r$ and $\eta^h = 0$
on $S_2\setminus\Sigma_r$; we choose it such that
\begin{equation}
\label{etabd}
\|\D'\eta^h\|_{L^{\infty}(\Sigma_r)}\leq \frac{C}{r}.
\end{equation}
Set $\eta_1^h = \eta^h$ and $\eta_2^h = 1 - \eta^h$.
Let $d\in W^{1, \infty}(S, \R^3)$ and define the recovery sequence
\begin{equation}
\label{ladefiniciondey}
y^h(x', x_3) = \sum_{i = 1}^2
\left(u_i^h(x') + hx_3 n_i^h(x') \right)\eta_i^h(x')
+ h^2\frac{x_3^2}{2}d(x').
\end{equation}
For later use we note that by this definition
\begin{align*}
\left|y^h - \sum_{i=1}^2u_i\eta_i^h\right|&\leq
h(|n^h_1| + |n_2^h|) + h^2|d|. 
\end{align*}
Hence, in view of \eqref{uppera-2} we conclude that
there is a constant $C$ depending only on $S$, $\Sigma$ and $y$ such that
\begin{equation}
\label{ladefi-1}
|y^h - y|\leq Ch\left(1 + h|d|\right)\mbox{ on }(S\setminus\Sigma_r)\times I.
\end{equation}
Now we compute
\begin{align*}
\D' y^h(x', x_3) &=\sum_{i = 1}^2 \left(\D' u_i^h(x') + hx_3 \D' n_i^h(x') \right)\eta_i^h(x')
\\
&+ \sum_{i = 1}^2 \left(u_i^h(x') + hx_3 n_i^h(x') \right)\D'\eta_i^h(x') 
+ h^2\frac{x_3^2}{2}\D' d(x').
\end{align*}
Recalling that $\eta_1^h = \eta^h$ and $\eta_2^h = 1 - \eta^h$,
we see that on $\Sigma_r\times I$
\begin{equation}
\label{dybd-1}
\begin{split}
|\D' y^h| &\leq \sum_{i = 1}^2 (|\D' u_i^h| + h|\D' n_i^h|)  \\
& \quad + |u_1^h - u_2^h||\D'\eta^h| + h |n_1^h - n_2^h||\D'\eta^h|
+ h^2|\D' d|
\\
&\leq C\big(
1 
+ \frac{1}{r}|u_1^h - u_2^h| + \frac{h}{r} |n_1^h - n_2^h|
\big).
\end{split}
\end{equation}
We have used the bound \eqref{etabd} as well as \eqref{uppera-2} 
and the estimate \eqref{uppera-1} for $\D' n_i^h$.
Similarly, since
$$
\d_3 y(x', x_3) = h \sum_{i = 1}^2 n_i^h(x')\eta_i^h(x') + h^2 x_3 d(x'),
$$
we can estimate
\begin{equation}
\label{dybd-2}
\frac{1}{h}|\d_3 y^h| 
\leq\sum_{i = 1}^2 |n_i^h\eta_i^h| + h|x_3 d|
\leq C(|n_1^h| + |n_2^h| + h) \leq C,
\end{equation}
in view of \eqref{uppera-2}.
Recalling \eqref{parameter-2},
we deduce from \eqref{dybd-1} and \eqref{dybd-2} that
\begin{equation}
\label{uppera-3}
|\D_h y^h| \leq C\big(1 + \frac{1}{r}|u_1^h - u_2^h| \big)
\mbox{ on }\Sigma_r\times I.
\end{equation}
Here we used \eqref{uppera-2} to estimate $|n_1^h - n_2^h|$
on the right-hand side of \eqref{dybd-1}.
We claim that 
\begin{equation}
\label{uppera-4}
|u_1^h - u_2^h|\leq Cr\mbox{ on }\Sigma_r.
\end{equation}
To prove this, note that \eqref{uppera-2}, \eqref{rhobd} and \eqref{parameter-2}
imply that
$$
|u_i^h - u_i| \leq C\rho_h\leq Ch\leq Cr\mbox{ on }S.
$$
Hence it remains to show that
\begin{equation}
\label{uppera-4b}
|u_1 - u_2|\leq Cr\mbox{ on }\Sigma_r.
\end{equation}
But $u_1 - u_2$ is Lipschitz on $U$ in view of \eqref{extended}.
Moreover it is zero on $\Sigma$ because $y$ is continuous.
Hence \eqref{uppera-4b} follows from the definition of $\Sigma_r$.
This concludes the proof of \eqref{uppera-4}.

By \eqref{uppera-4} and \eqref{uppera-3} we see that
\begin{equation*}
\|\D_h y^h\|_{L^{\infty}(\Sigma_r\times I)} \leq C.
\end{equation*}
Since $W$ is locally bounded, this implies
$$
\|W\big(\D_h y^h\big)\|_{L^{\infty}(\Sigma_r\times I)} \leq C.
$$
Therefore, since $|\Sigma_r|\leq Cr$ due to the regularity of $\Sigma$,
\begin{equation}
\label{ende-A}
\frac{1}{h^2}E^h(y^h, \Sigma_r\times I) =
\frac{\e}{h^2}\int_{\Sigma_r\times I} W(\D_h y^h) \dv{x}
\leq 
\frac{C\e r}{h^2}.
\end{equation}
The right-hand side converges to zero due to \eqref{parameter-3}.

On $S_i\setminus\Sigma_r$ the function $\eta_i^h$ is identically equal to $1$.
Hence
\begin{equation}
\label{Gradiy}
\D_h y^h = 
(\D' u^h_i\ |\ n^h_i) + hx_3(\D' n^h_i\ |\ d) + \frac{h^2x_3^2}{2}(\D' d\ |\ 0)
\mbox{ on }(S_i\setminus\Sigma_r)\times I.
\end{equation}
The map $R = (\D'y\ |\ n)$ clearly takes values in $SO(3)$.
Define
$$
a_i^h =   
x_3 R^\transp(\D' n^h_i\ |\ d) + \frac{hx_3^2}{2} R^\transp(\D' d\ |\ 0).
$$
Then
\begin{equation}
\label{aihb-20}
|a_i^h| \leq C(|\D' n^h_i| + |d| + h|\D' d|)\leq C(1 + |\D' n^h_i|).
\end{equation}
Hence \eqref{uppera-1} ensures that, for small $h$,
\begin{equation}
\label{aihb}
h|a_i^h|\leq C\mbox{ on }S.
\end{equation}
On $S\setminus M^h$ we have
$n_i^h = n_i$, hence
$\D'n_i^h = \D' n_i$ almost everywhere on this set.
Therefore, \eqref{aihb-20} shows that
\begin{equation}
\label{aihb-2}
|a_i^h|\leq C(1 + |\D' n_i|)\mbox{ on }S_i\setminus M^h.
\end{equation}
By the frame indifference of $W$ we have, almost everywhere on 
$(S_i\setminus\Sigma_r)\times I$,
\begin{equation}
\label{fjm1-1}
W(\D_h y^h) = 
W\big(R^\transp\D_h y^h\big) = W\big(R^\transp(\D' u_i^h\, |\, n_i^h) + ha_i^h\big).
\end{equation}
On $S_i\setminus M_h$ we have $(\D' u^h_i\, |\, n^h_i) = R$,
so on $(S_i\setminus\Sigma_r\setminus M^h)\times I$
\begin{align*}
\frac{1}{h^2}W(\D_h y^h) &= 
\frac{1}{h^2}W(I + ha_i^h)
\\
&\leq \frac{C}{h^2}\dist^2_{SO(3)}(I + ha_i^h)
\leq C|a_i^h|^2.
\end{align*}
We have used \eqref{aihb} and the fact 
that the hypotheses on $W$ imply that $W\leq C\dist^2_{SO(3)}$
on bounded subsets of $\R^{3\times 3}$.
Now \eqref{aihb-2} implies the bound
\begin{equation}
\label{aihb-3}
\frac{1}{h^2}\chi_{S_i\setminus M_h} W(I + ha_i^h) \leq C(1 + |\D' n_i|^2).
\end{equation}
The right-hand side is in $L^1(S_i)$ and does not
depend on $h$. On the other hand, by Taylor expansion
and since $|\Sigma_r|\to 0$ and $|M^h|\to 0$
$$
\frac{1}{h^2}\chi_{(S_i\setminus\Sigma_r\setminus M_h)\times I} W(I + ha_i^h)
\to
\frac{1}{2} Q_3\big(x_3 R^\transp (\D' n_i\, |\, d) \big)
$$
pointwise almost everywhere on $S_i\times I$. Combining this
with \eqref{aihb-3} we can apply dominated convergence to conclude
\begin{equation}
\label{ende-B}
\frac{1}{h^2}E^h(y^h, (S_i\setminus\Sigma_r\setminus M_h)\times I)
\to \frac{1}{24}\int_{S_i} 
Q_3(R^\transp (\D' n_i\, |\, d)) \dv{x'}.
\end{equation}
We now claim that
\begin{equation}
\label{nachtrab-1}
\limsup_{h\to 0} \|W\big(\D_hy^h\big)\|_{L^{\infty}\left((S\setminus\Sigma_r)\times I\right)} < \infty.
\end{equation}
In fact, since $W$ is locally bounded, \eqref{nachtrab-1}
will follow once we show that
\begin{equation}
\label{finalmente-b}
\|\D_hy^h\|_{L^{\infty}((S\setminus\Sigma_r)\times I)}\leq C.
\end{equation}
But by \eqref{Gradiy}, on $(S_i\setminus\Sigma_r)\times I$ we have
\begin{align*}
|\D_hy^h|\leq C\left(1 +
|\D'u_i^h| + |n_i^h| + h|\D' n_i^h|\right).
\end{align*}
The last term on the right-hand side is uniformly bounded due to \eqref{uppera-1},
whereas the other two are uniformly bounded due to \eqref{uppera-2}.
This concludes the proof of \eqref{nachtrab-1}.
\\
Using \eqref{nachtrab-1} we see that
\begin{equation}
\label{ende-C}
\frac{1}{h^2}E^h(y^h, (M_h\setminus\Sigma_r)\times I)
= \frac{1}{h^2}\int_{(M^h\setminus\Sigma_r)\times I} W(\D_hy^h) \leq
\frac{C}{h^2} |M_h|.
\end{equation}
By \eqref{mhbd} the right-hand side converges to zero as $h\to 0$.
Summarizing, by combining \eqref{ende-A}, \eqref{ende-B} and \eqref{ende-C} we see that
$$
\frac{1}{h^2}E^h(y^h)
\to \frac{1}{24}\int_S Q_3(R^\transp (\D' n\, |\, d)) \dv{x'},
$$
where $n$ is the normal to $y$.
Relaxing over $d\in L^2$ as in \cite{FrJaMu02}, the convergence
\eqref{gammalimsup-energies} follows.
\\
More precisely, there exist $d_j\in W^{1,\infty}(S, \R^3)$ converging strongly
in $L^2$ and a sequence $h_j\to 0$ such that, defining $y^{h_j}$
as in \eqref{ladefiniciondey} with $d = d_j$,
the convergence \eqref{gammalimsup-energies} is true for $h = h_j$.
Proposition \ref{lower&compact} implies that, after taking a subsequence,
$(y^{h_j})$ converges weakly in $W^{1,2}(\Omega)$ to some 
$\widetilde y\in\AA(S, \Sigma)$.
Since the $d_j$ remain uniformly bounded in $L^2$,
estimate \eqref{ladefi-1} ensures that $\widetilde y = y$.
\end{proof}

\section{Discretization}\label{sec:fem}
We devise in this section a discretization of
the folding problem based on the use of an isoparametric
discontinuous Galerkin finite element method. Corresponding
functions and related discrete quantities are marked by a tilde
sign. 

\subsection{Finite element spaces} 
We follow~\cite{BoNoNt21} and let $\tcT$ be a partition of
the Lipschitz domain $S$ into closed, shape regular elements
$T\in \tcT$ which are images of mappings $F_T: \hT \to T$, where
$\hT$ is a reference triangle or square. The space $\tVV$ of
discontinuous piecewise transformed polynomials of fixed polynomial
degree $k \ge 2$ is defined as 
\[
\tVV = \{\tV \in L^2(S):
\tV \circ F_T \in P_k \cup Q_k \, \mbox{ for all } T\in \tcT\},
\]
where $P_k$ and $Q_k$ denote polynomials of total and partial
degree~$k$ on the respective reference element. We let
$\tcE^{int}$ be the set of interior edges.

The elementwise application of a differential operator is indicated
by a tilde, e.g., for $\tV\in \tVV$ we define
\[
\tnabla \tV|_T = \nabla (\tV|_T)
\]
for all $T\in \tcT$. We use standard notation to denote jumps and 
averages of elementwise smooth functions, e.g.,
\[
\jump{\tV}_e = \tV^+ - \tV^-, \quad
\aver{\tV}_e = (\tV^+ + \tV^-)/2,
\]
for an inner side $e= T^+ \cap T^-$ with a fixed unit normal
$\mu_e$ pointing from $T^+$ into $T^-$ that determines the sign
of $\jump{\tV}$. 

To match the targeted experiments, the boundary conditions imposed in all the numerical simulations provided below 
are pointwise Dirichlet boundary conditions, i.e. we enforce $y(x_i^D) = g_i$, where $x_i$ is a vertex of the 
subdivision on the boundary of $S$ and $g_i$ are given boundary deformations, $i=1,...,n_D$. 
Whence, the jump and average operators do not need to be defined on boundary edges as in the free boundary 
case \cite{BoGuNoYaTh} unlike the clamped boundary case \cite{BoNoNt21}.

\subsection{Curve approximation}
We assume throughout that the folding curve $\Sigma$ is Lipschitz continuous
and piecewise $C^2$ with possible kinks only occurring at vertices of the subdivision.
Moreover, we assume that either a parametric description 
$\Sigma = \{ \gamma(u), u\in [0,1] \}$ or, provided that $\Sigma$ is $C^2$, the 
distance $d_{\Sigma}$ to the curve is available.  We also assume that the
triangulation defines a piecewise smooth curve 
\[
\tSigma = \cup_{j=1}^J e_j
\]
with inner sides $e_j\in \tcE^{int}$, $j=1,\dots,J$, such that the endpoints 
of the segments $e_j$ belong to $\Sigma$. This implies that there exists 
a bijection $\tM: \tSigma \to \Sigma$ such that the distance between the two curves is 
small in the sense that
\begin{equation}\label{e:M}
\|\tM - \id  \|_{W^{1,\infty}(\tSigma)} \to 0
\end{equation}
as $h \to 0$.

\subsection{Discrete Hessian}
To obtain good consistency properties of the approximate Hessian $\tH(\tV)$
of a function $\tV\in \tVV$ we first note that the distributional Hessian
$\mathrm{D}^2\tV$ is for $\phi \in C^\infty_c(S\setminus \tSigma;\R^{2\times 2})$ on the open
set $S \setminus \tSigma$ given by
\[\begin{split}
\dual{\mathrm{D}^2\tV}{\phi} &= \int_{S\setminus \tSigma} \tV \diver \Diver \phi \dv{x'} \\
&= \int_{S\setminus \tSigma} \tD^2 \tV : \phi \dv{x'} 
+ \sum_{e\in \tcE^{int}\setminus \tSigma} \int_e \jump{\tV} \big(\Diver \phi \cdot \mu_e \big)
-  \jump{\tnabla \tV} \cdot \big(\phi \mu_e\big) \dv{s},
\end{split}\]
where $\Diver$ denotes the application of the standard divergence operator
to the columns of a matrix-valued function. 
We aim at preserving this identity for elementwise polynomial
functions $\tphi$ and represent the contributions on  the interior edges $\tcE^{int}$
by functions defined in the edge patches $\o_e = T^- \cup T^+$.
We follow ideas from~\cite{Pryer14,BoNoNt21} and define the operators 
\[
s_e : L^2(e;\R) \to \tVV^{2\times 2}|_{\o_e}, \quad
r_e : L^2(e;\R^2) \to \tVV^{2\times 2}|_{\o_e}, 
\]
for inner edges $e\in \tcE^{int}$ via
\[\begin{split}
\int_{\o_e} s_e(\hv) : \tphi \dv{x'} &= \int_e \hv \big\{\Diver_h \tphi \cdot \mu_e\big\}  \dv{s}, \\
\int_{\o_e} r_e(\hw) : \tphi \dv{x'} &= \int_e \hw \cdot \big\{\tphi \mu_e\big\} \dv{s},
\end{split}\]
for all $\tphi\in \tVV^{2\times 2}|_{\o_e}$; the functions $s_e(\hv)$ and $r_e(\hw)$ are trivially
extended to $S$. We then define $\tH(\tV)\in \tVV^{2\times 2}$ as
\[
\tH(\tV) = \tD^2 \tV + S_\cE(\tV) - R_\cE(\tnabla \tV),
\]
where 
\[
S_\cE(\tV) =  \sum_{e\in \tcE^{int}} s_e(\jump{\tV}), \quad
R_\cE(\tnabla \tV) = \sum_{e\in \tcE^{int} \setminus \tSigma} r_e(\jump{\tnabla \tV}).
\]
Note that the contributions $S_\cE$ associated with the continuity of $y$ contains the edges on
$\tSigma$ while these are omitted in $R_\cE$ respecting possible discontinuities in  
deformation gradients.

For every $\tphi\in \tVV^{2\times 2} \cap C_c^1(S\setminus \tSigma;\R^{2\times 2})$ we have
the consistency property
\[
\dual{\mathrm{D}^2\tV}{\tphi} = \int_U \tH(\tV) : \tphi \dv{x'}.
\]

In general the intersection 
$\tVV^{2\times 2} \cap C_c^1(S\setminus \tSigma;\R^{2\times 2})$ only
contains constant functions. When the interface $\Sigma$ is exactly captured by the subdivisions, 
i.e. $\tSigma = \Sigma$, then the reconstructed Hessian restricted to any subdomain separated 
by $\Sigma$ weakly converge to the continuous Hessian in $L^2$ \cite{BoGuNoYaTh}. We define 
a discrete seminorms approximating a seminorm of $W^{2,2}(S\setminus \Sigma) \cap W^{1,2}(S)$
for $\tV\in \tVV$ via
\[
\|\tV\|_{\tH^2}^2 = \|D_h^2 \tV\|^2
+ \int_{\tcE^{int}} h_\tcE^{-3} |\jump{\tV}|^2 \dv{s}
+ \int_{\tcE^{int} \setminus \tSigma} h_\tcE^{-1} |\jump{\tnabla \tV}| \dv{s}.
\]
Note that 
the 
identity $\|\tV\|_{\tH^2} =0$ only implies that
$\tV$ is continuous and piecewise affine.
By using standard inequalities we find that the discrete Hessian
defines a bounded operator in the sense that
\[
\|\tH(\tV)\| \le c \|\tV\|_{\tH^2}
\]
for all $\tV\in \tVV$ with a constant $c>0$ that is independent of the
cardinality of $\tcT$.

\subsection{Discrete energy functional}
Our discrete energy functional is defined on a discrete admissible set
that enforces the isometry condition up to a tolerance $\tvrho>0$, i.e.,
we set
\[
\tcA = \Big\{\tZ \in \tVV^3: 
\sum_{T\in \tcT} \Big| \int_T (\tnabla \tZ)^\transp (\tnabla \tZ) - I \dv{x'} \Big| 
\le \tvrho \Big\}.
\]
The discrete functional $\tE_K$ is then obtained by replacing the
Hessian by its discrete approximation which is applied componentwise
and introducing stabilizing and penalty
terms, i.e., for $\g_0,\g_1, \g_2>0$ and $\tY\in \tcA$ we set
\[\begin{split}
\tE_K(\tY) =&  \frac{1}{24} \int_{S\setminus \tSigma} |\tH(\tY)|^2 \dv{x'} \\
& + \frac{\g_0}{2} \int_{\tcE^{int}} h_{\tcE}^{-3} |\jump{\tY}|^2 \dv{s}
+ \frac{\g_1}{2} \int_{\tcE^{int} \setminus \tSigma} h_{\tcE}^{-1} |\jump{\tnabla \tY}|^2 \dv{s} \\
& + \frac{\gamma_2}2 \sum_{i=1}^{n_D}h_i^{-2} \lbrack (\tY-g_i)(x_i^D)\rbrack^2.
\end{split}\]
Assuming an isotropic material we have up to a constant factor $Q(A)=|A|^2$;
we note that the approach applies to more general quadratic forms. 
Note that unlike in previous works, pointwise Dirichlet conditions are considered and enforced
via penalization; $h_i$ denotes a local meshsize around the vertex $x_i^D$.

The energy functional $\tE_K$ is uniformly coercive in $\tH^2$,
i.e., there exists a constant $c$ such that for any choice of parameters $\g_0,\g_1,\g_2>0$ 
we have for all $\tY \in \tVV^3$
$$
\|\tY\|_{\tH^2} \leq c_1 \tE_K(\tY), \qquad \tY \in \tVV^3.
$$
Furthermore, the gradients of deformations in the discrete admissible set $\tcA$ are  uniformly bounded:
$$
\|\tnabla \tY\| 
\le \sqrt{2} \left(\tvrho + |S|\right), \qquad \forall \tY \in \tcA.
$$
We refer to \cite{BoGuNoYaTh} for proofs of the above two inequalities.
Note that these two estimates do not provide a uniform $L^2$ control.
As a consequence, depending on the boundary conditions, the deformations may be defined up 
to certain invariances. If the gradient flow metric controls the $L^2$ norm, then the 
discrete gradient provides unique iterates.

A rigorous justification of the discrete energy functional $\tE_K$ 
can be obtained by establishing its Gamma convergence to $E_K$ as 
the maximal mesh-size~$h$ tends to zero. To prove the stability bound or 
liminf inequality 
one uses the coervity estimate and follows~\cite{BoNoNt21,BoGuNoYaTh} to show by 
using regularizations obtained with
quasi-interpolation operators that for a sequence $(\tY)_{h>0}$
with $\tE_K(\tY)\le c$ there exists a subsequence and a limit
$y\in W^{2,2}(S\setminus \Sigma;\R^3)\cap W^{1,\infty}(S;\R^3)$ such 
that in $L^2(S)$ we have for $h\to 0$
\[
\tY \to y, \quad \tnabla \tY \to \nabla y, \quad \tH(\tY) \wto D^2 y,
\]
provided that elements $T\in \tcT$ satisfy a geometric condition away from the 
discrete interface $\tSigma$. The consistency or limsup inequality
requires the construction of suitable interpolants $\tcI y \in \tcA$ 
of a given folding isometry $y\in W^{2,2}(S\setminus \Sigma;\R^3)\cap W^{1,\infty}(S;\R^3)$
such that 
\[
\tE_K(\tcI y) \to E_K(y)
\]
as $h \to 0$. Crucial here is to show that on curved elements $T\in \tcT$
along the discrete folding arc $\tSigma$ the difference of the local
energy contributions 
\[
\Big| \int_T |\tD^2 \tcI y|^2 \dv{x'}  - \int_{\hT} |D^2 y|^2 \dv{x'} \Big|
\]
with the corrected element $\hT$ hat has a side on the exact interface
can be sufficiently controlled. Corresponding details are in preparation.

\section{Numerical experiments}\label{sec:num_exp}
We report in this section on numerical results obtained with the proposed 
numerical method and the iterative scheme. 

\subsection{Algorithmic aspects}

Except for the presence of folding curves and correspondingly removed
edge contributions in the discontinuous Galerkin method the overall
strategy follows closely the algorithm devised in~\cite{BoGuNoYaComp} 
and later analyzed in~\cite{BoGuNoYaTh}. The efficiency of the discrete
gradient flow~\eqref{e:gradient_flow} for finding stationary configurations
depends strongly on the availability of a good starting value, in particular
on its discrete energy $\tE_K^0$ and the isometry violation $\tvrho$, see~\eqref{e:discrete_admissible_set}.
We note that the boundary conditions are included in a weak, penalized form and, in practice, 
constitute a major contribution of the initial energy when the initial deformation 
is not suitably constructed. To obtain an initial deformation with simultaneously 
moderate discrete bending energy $\tE_K^0$ and small isometry violation $\tvrho$, 
we use the preprocessing procedure described in~\cite{BoGuNoYaComp}. It combines 
the solution of a linear bi-harmonic problem to obtain an approximate discrete
extension $\hY^0 \in \tVV^3$ of the boundary data with a subsequent gradient 
descent applied to the isometry violation error with an iteration until this 
quantity is below a given tolerance, i.e., until the iterate $\hY^L \in \tVV^3$
satisfies
\[
\frac{1}2\int_S | (\tnabla \hY^L)^\transp (\tnabla \hY^L) -I|^2 \dv{x} \le \veps_{pp}.
\]
We then define $\tY^0=\hY^L$ as the starting value for the 
gradient scheme~\eqref{e:gradient_flow}. The gradient flow metric $(\cdot,\cdot)_*$
is obtained as a combination of the bilinear form defined by the discrete 
energy functional and the $L^2$ norm. With this choice we avoid nonuniqueness effects
for certain boundary conditions. As a stopping criterion for this iteration 
we impose the condition that the discrete bending energy is nearly stationary, i.e.,
\[
\big| d_t \tE_K (\tY^M) \big| = \frac{\left| \tE_K(\tY^M)-\tE_K(\tY^{M-1}) \right| }{\tau} 
\le \veps_{stop},
\]
for a given tolerance $\veps_{stop}>0$. The deformation $\tY^M \in \tVV^3$ serves
as our approximation of stationary, low energy configuration for $\tE_K^0$ in the 
admissible set $\tcA$. 
Unless specified otherwise, piecewise polynomials of degree~2 are used for the 
approximation of the deformation, the lifting operators in the construction
of the discrete Hessian, and in approximating the folding curve by edges of 
elements. Our subdivisions are generated with the package \verb|Gmesh|
\cite{geuzaine2009gmsh}, the implementations make use of the \verb|deal.ii| library \cite{bangerth2007}, 
and the visualization are obtained using \verb|Paraview| \cite{squillacote2007paraview}. 
 
\subsection{Bistable flapping device}
Our first set of experiment considers the setting sketched in 
Figure~\ref{fig:model_exp}. The precise parameters defining
the domain $S$ and the arc $\Sigma$ are as follows.

\begin{example}[Parabolic and circular arcs]\label{ex:arc}
For $S = (0,9.6)\times (0,15)$ we consider compressive boundary 
conditions of rate $s \in (0,1)$ imposed at the corners  
\[
x_D = (0,0), \quad x_D' = (9.6,0.0).
\]
Two choices of a folding arc $\Sigma\subset \overline{S}$ are addressed: \\
(a) Let $\Sigma$ be the quadratic curve connecting two boundary points 
$x_{\Sigma,j} \in \p S$, $j=1,2$, and passing through the apex $x_{\Sigma,A}$
given by 
\[
x_{\Sigma,1} = (0,2), \quad x_{\Sigma,2} = (9.6,2), \quad x_{\Sigma,A}=(4.8,6).
\]
(b) Let $\Sigma$ be the circular arc with end-points $x_{\Sigma,j} \in \p S$,
$j=1,2$, and circular midpoint $x_{\Sigma,M} \not \in \Sigma$ given by 
\[
x_{\Sigma,1} = (0,2), \quad x_{\Sigma,2} = (9.6,2), \quad x_{\Sigma,M}=(4.8,-2),
\]
i.e., with radius $r^2=(4.8)^2 + 4^2$. 
\end{example}

A typical triangulation with~556 elements together with an exact resolution 
of the parabolic arc defined in Example~\ref{ex:arc}~(a) 
is shown in Figure~\ref{fig:ex_arc_triang}. Note that the arc is matched exactly
by edges of elements. The simulations are performed for a 
pseudo-time step $\tau=0.01$ and tolerances  $\veps_{stop} = 0.01$, $\veps_{pp} = 1.0$. 
The numerical approximations $\tY^M \in \tVV^3$ obtained with the numerical
scheme for different compression rates imposed in the boundary points $x_D$ and 
$x_D'$ are shown in Figure~\ref{fig:ex_arc_stationary}. We observe a good qualitative
agreement with the real experiment shown in the left part of Figure~\ref{fig:model_exp}
and a continuous dependence of the deformation on the compression rate. 
Only $0,5,10,15$ iterations 
of the gradient descent method for compression rates $s=0\%, 10\%, 20\%,$ and $30\%$ 
were required to meet the prescribed stopping criterion. 

\begin{figure}[htb]
\includegraphics[width=0.5\textwidth,angle=0]{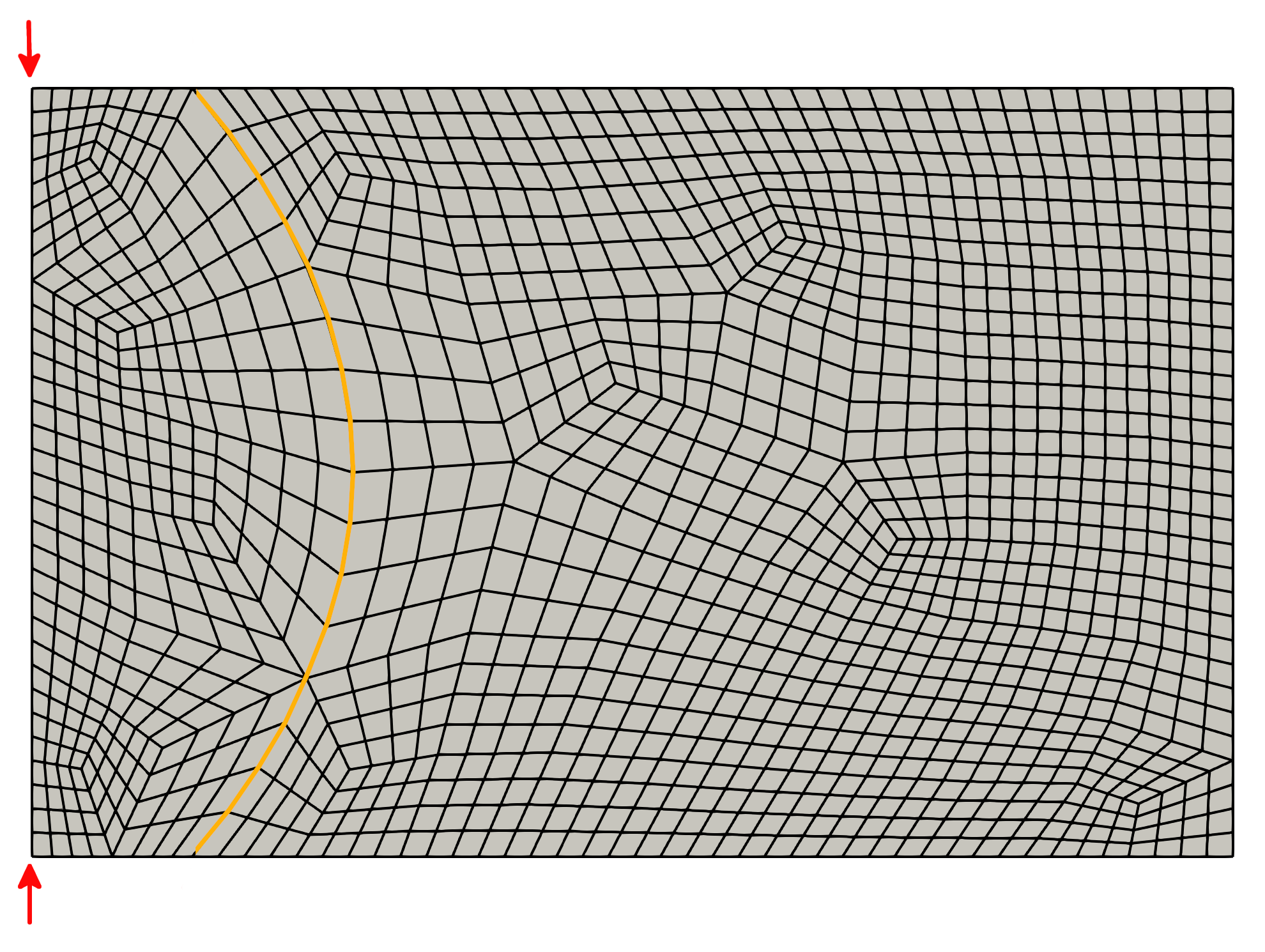}
\caption{\label{fig:ex_arc_triang} Triangulation, folding arc approximation, and compressive
point boundary conditions to generate a bistable flapping mechanism described
by Example~\ref{ex:arc} with a quadratic arc $\Sigma$ that is resolved exactly.}
\end{figure}

\begin{figure}[p]
\includegraphics[width=0.35\textwidth]{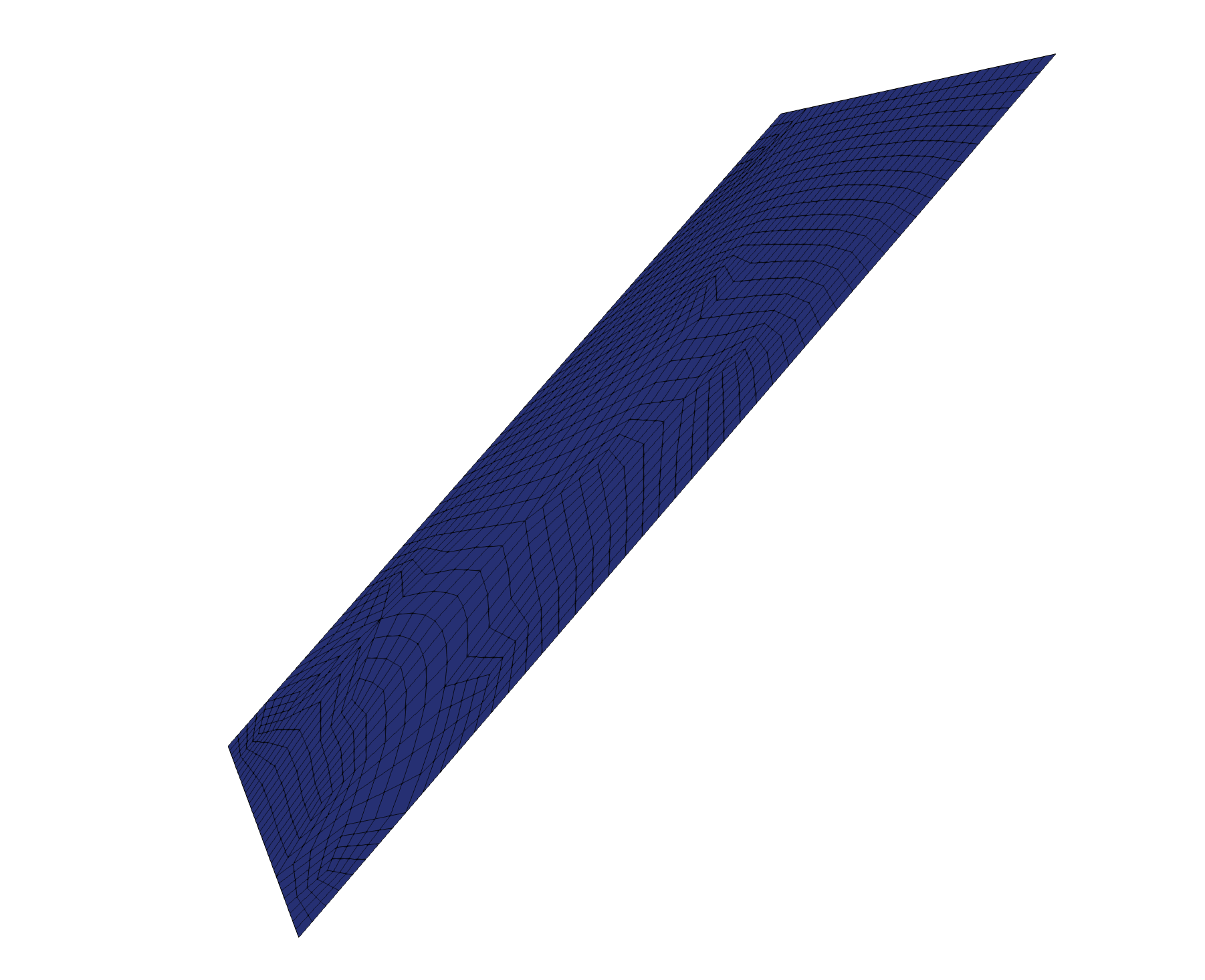} 
\includegraphics[width=0.35\textwidth]{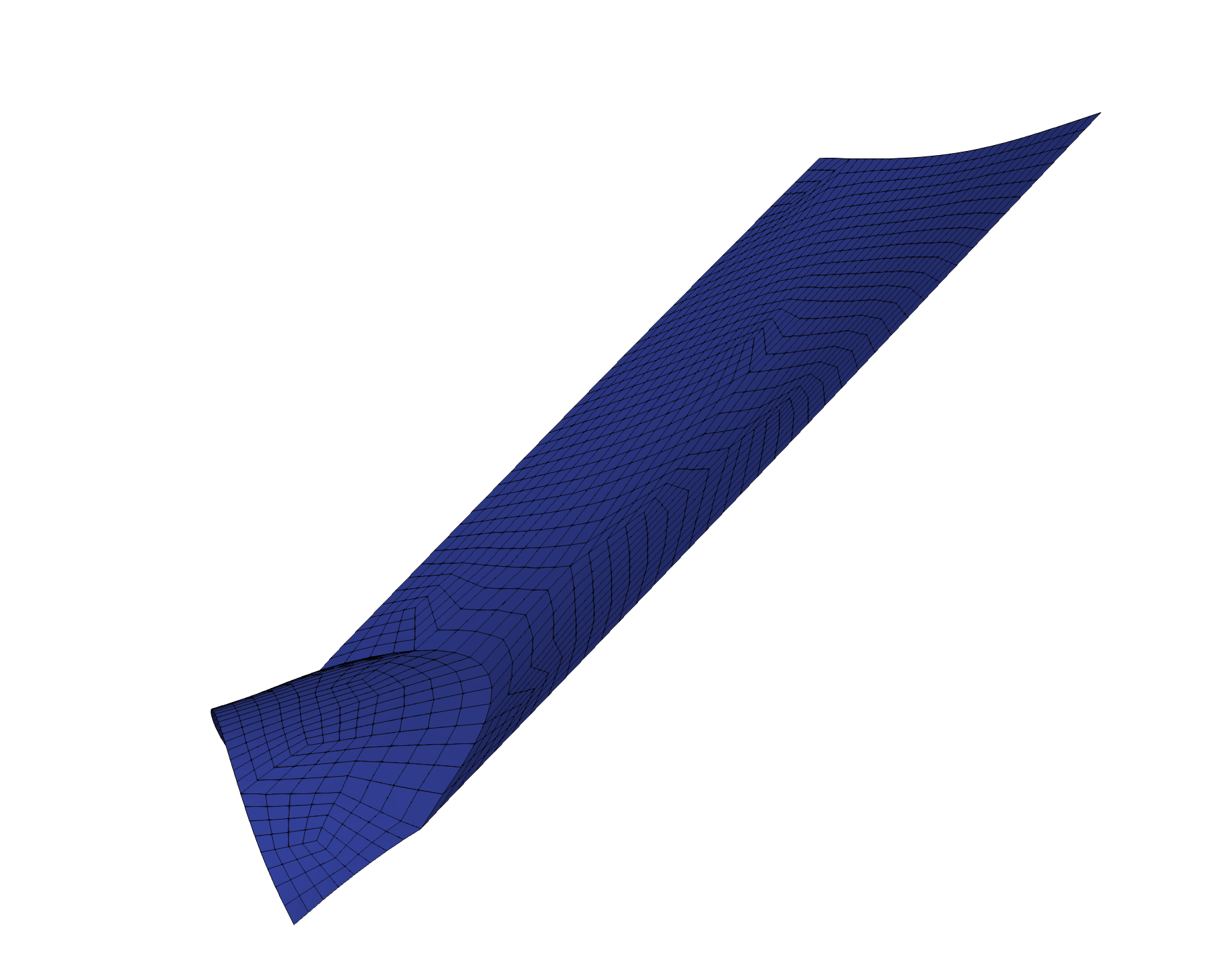} \\
\includegraphics[width=0.35\textwidth]{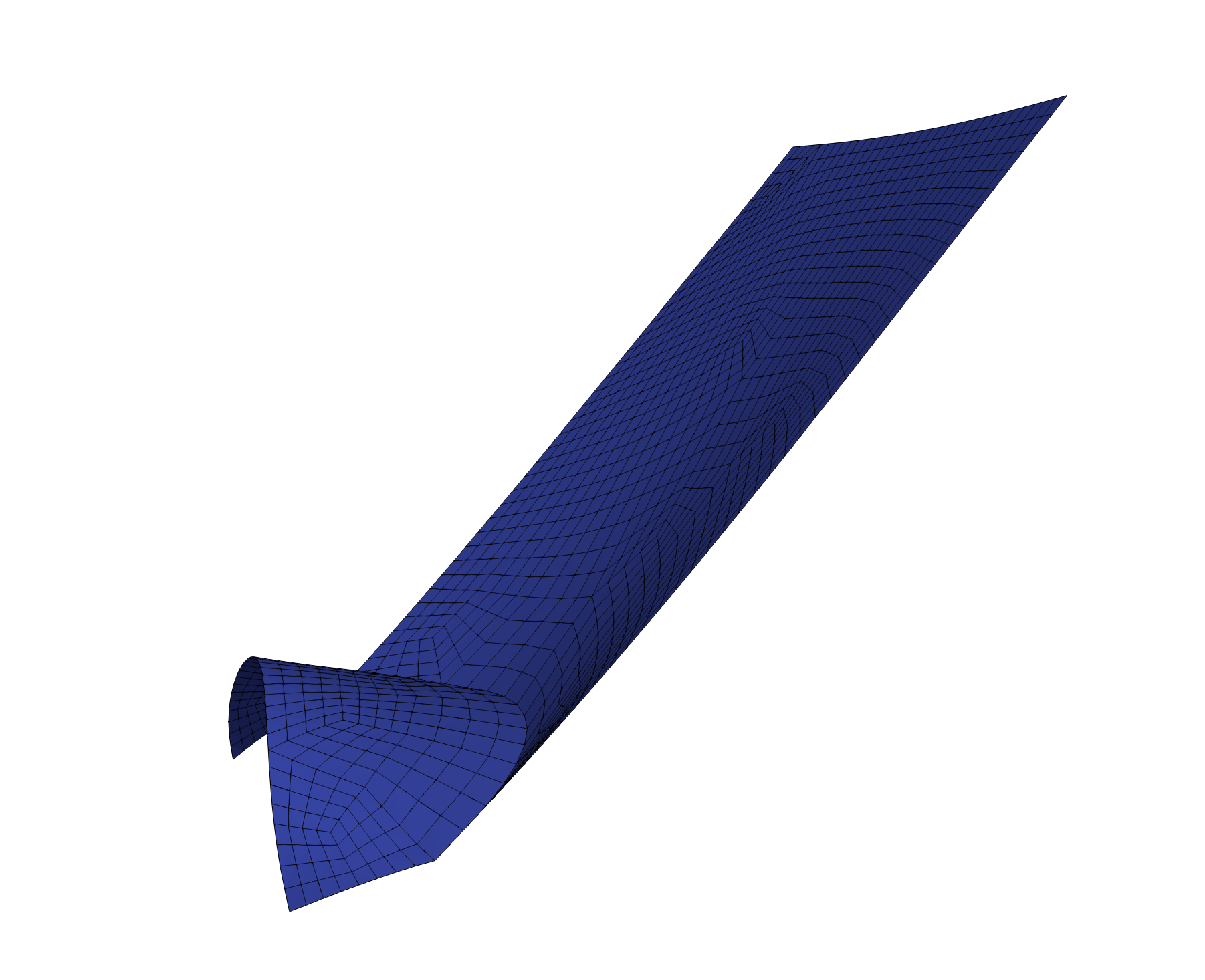}
\includegraphics[width=0.35\textwidth]{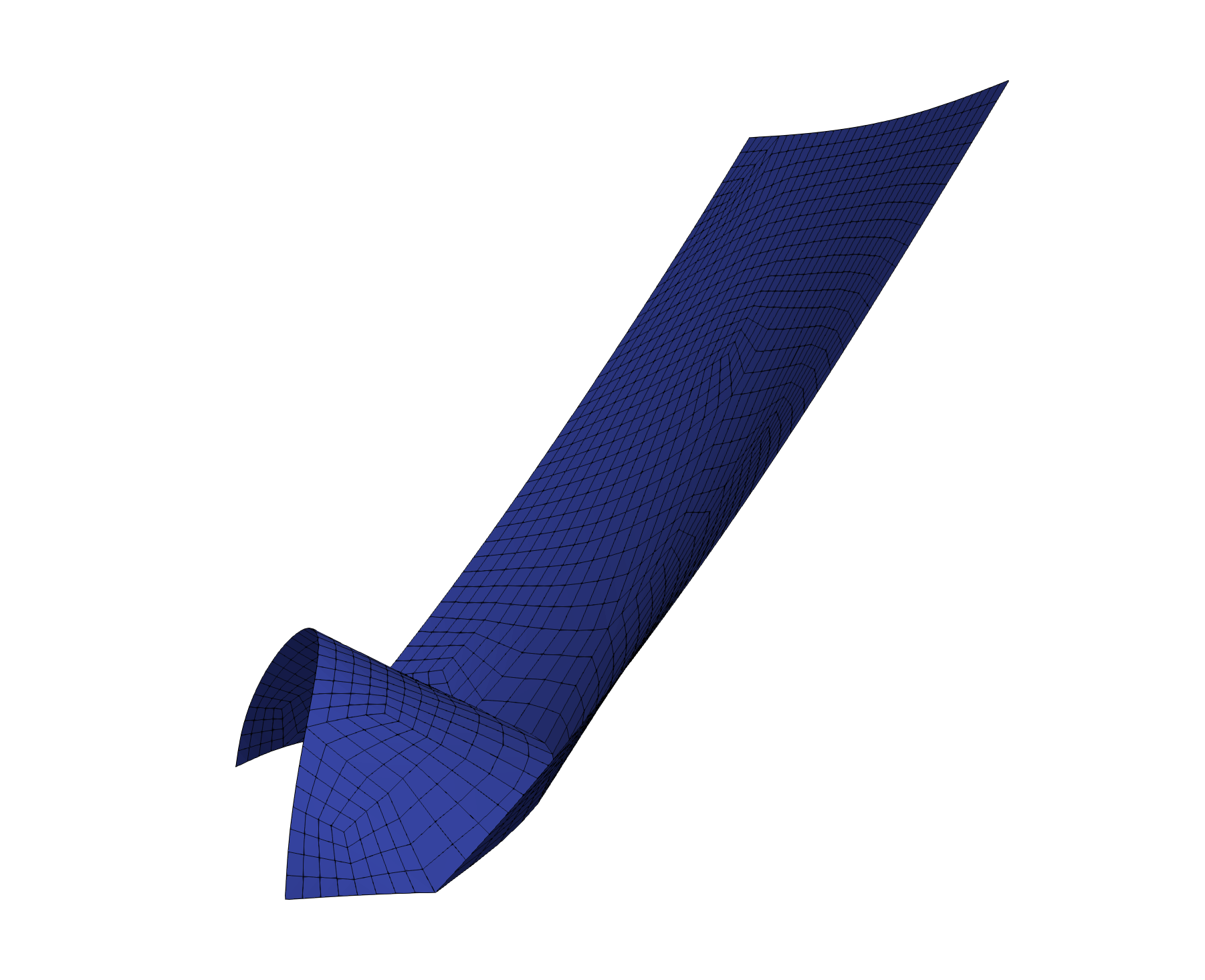}
\caption{\label{fig:ex_arc_stationary}
Nearly stationary configurations $\tY^M \in \tVV^3$ in Example~\ref{ex:arc}
with quadratic folding arc for compression rates $s= 0\%, 10\%$, $20\%$, 
and $30\%$ (left to right, top to bottom).}
\end{figure}

When the folding arc is circular instead of parabolic then our discrete
curves $\Sigma_h$ do not resolve the goemetry exactly. For the setting 
described in Example~\ref{ex:arc}~(b) and a triangulation consisting again of~556 
elements that provide a piecewise quadratic approximation $\Sigma_h$
of $\Sigma$ we obtained for the parameter choices $\tau=0.01$, $\veps_{stop}= 0.1$,
and $\veps_{pp}=1$ the nearly stationary configurations shown in 
Figure~\ref{fig:ex_circle_stationary}. The discrete deformations are similar
to those obtained for the parabolic arc except that the deformed 
right side of the initial rectangular plate is now curved. 

\begin{figure}[p]
\includegraphics[width=0.43\textwidth]{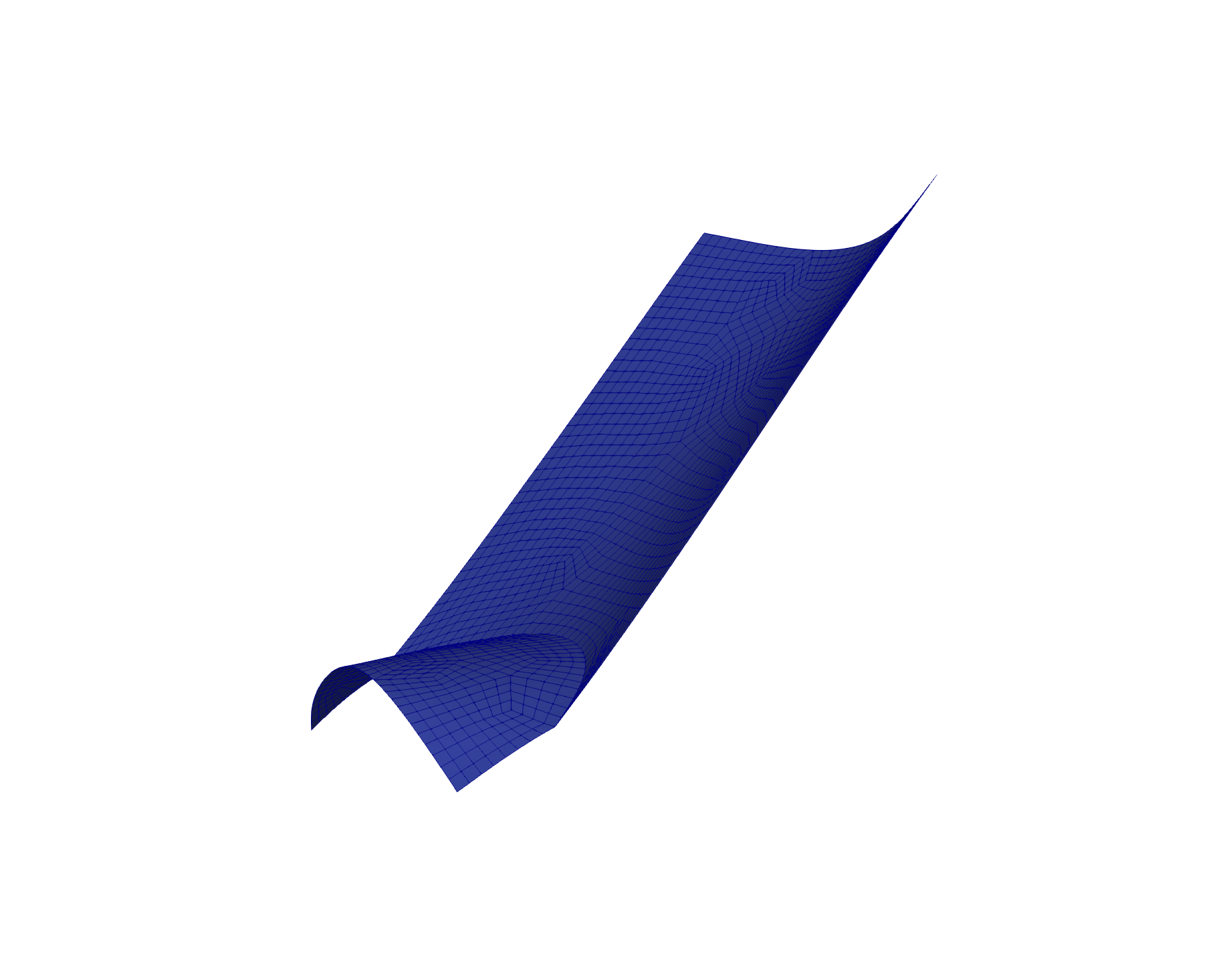} \hspace*{-8mm} 
\includegraphics[width=0.43\textwidth]{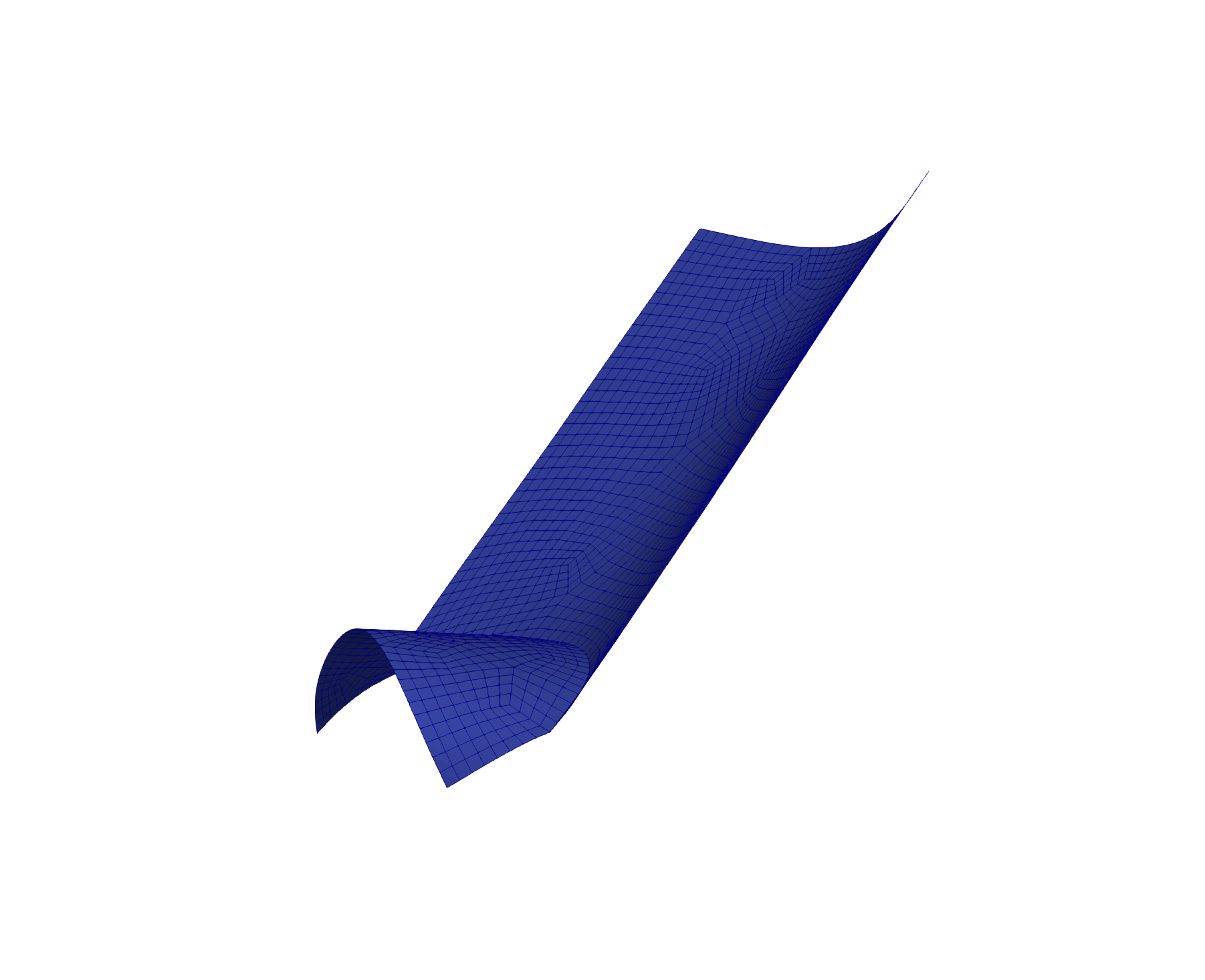} \\[-8mm]
\includegraphics[width=0.43\textwidth]{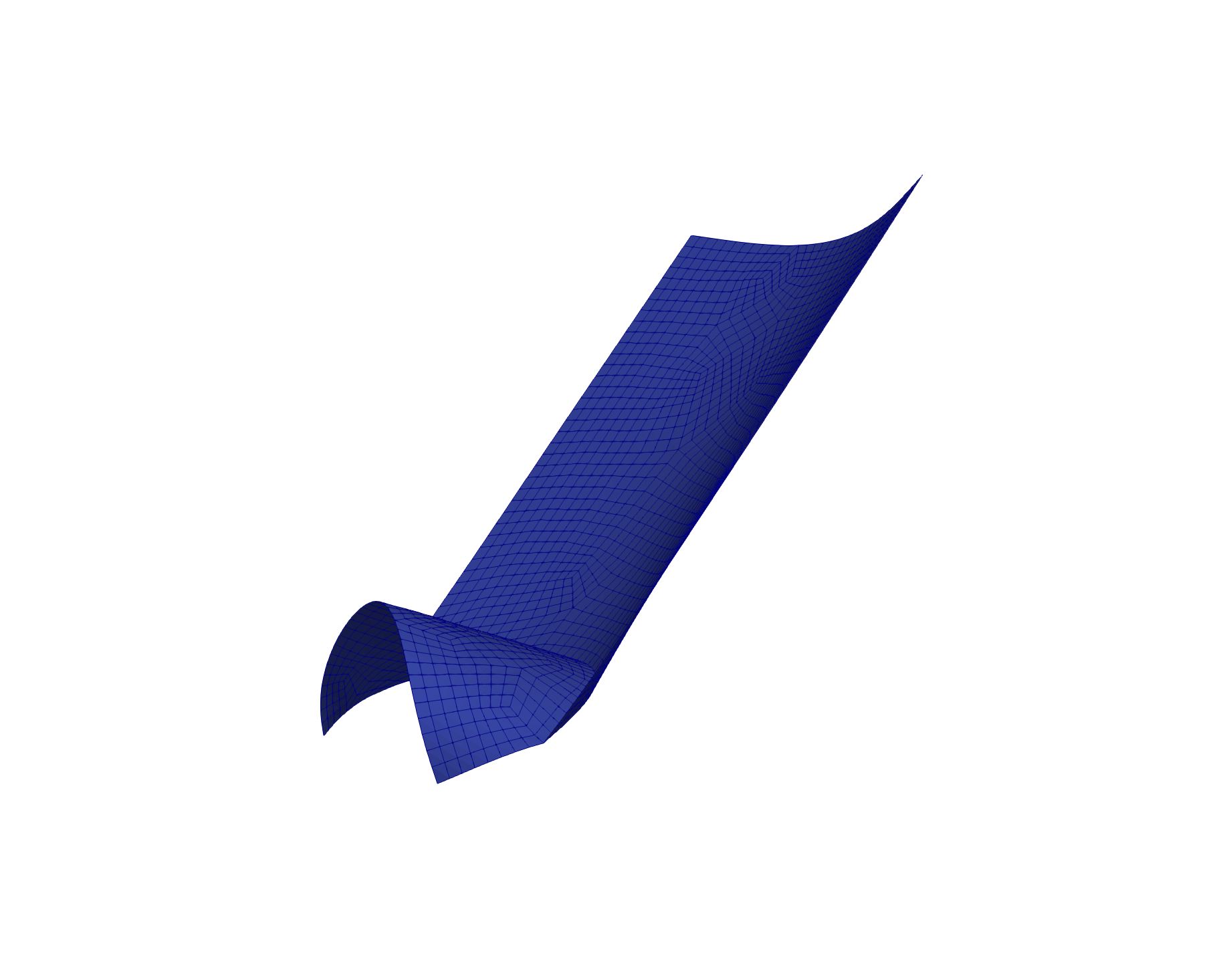} \hspace*{-8mm}
\includegraphics[width=0.43\textwidth]{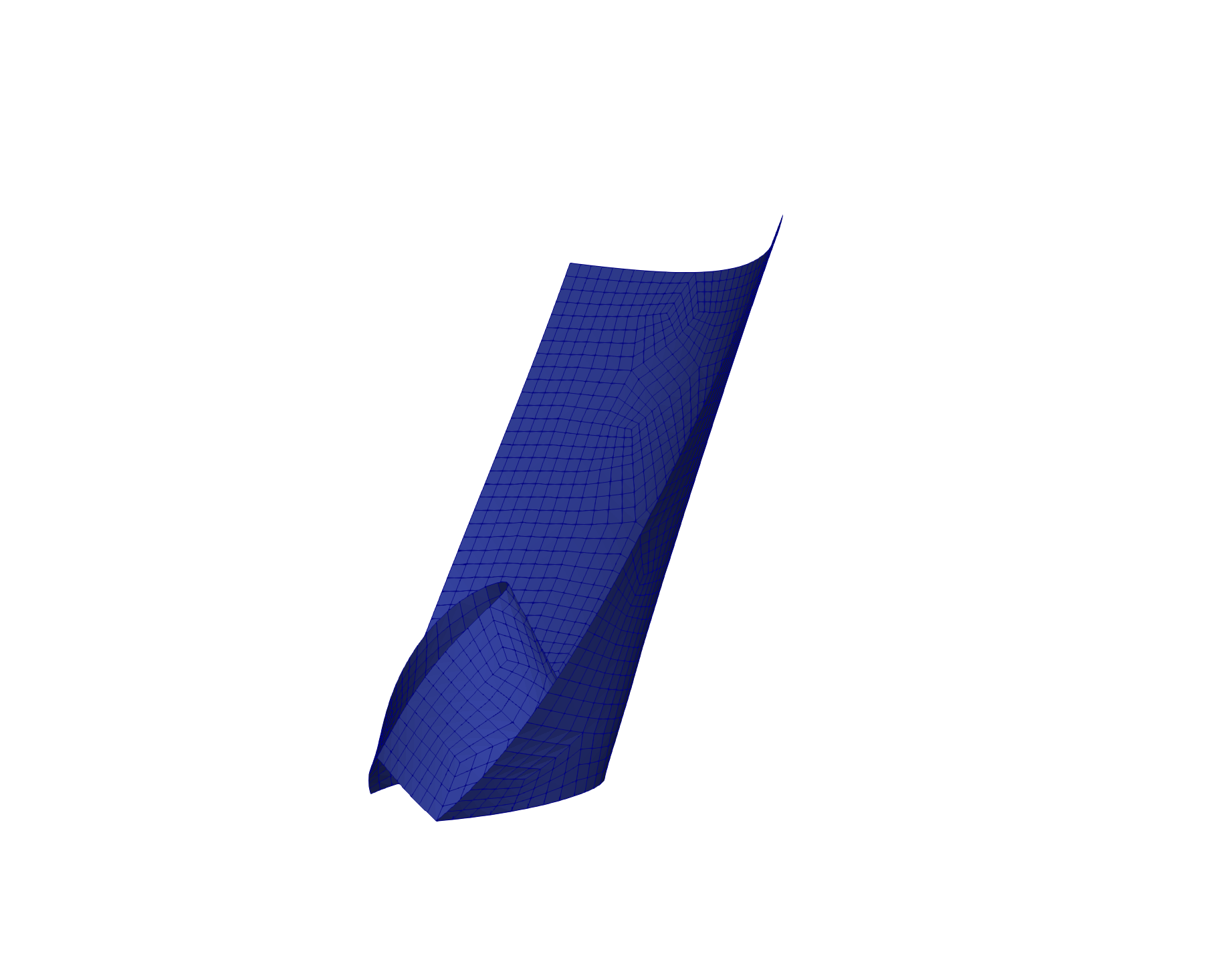}
\caption{\label{fig:ex_circle_stationary} 
Nearly stationary configurations $\tY^M \in \tVV^3$ in Example~\ref{ex:arc}
with circular folding arc for compression rates $s= 10\%$, $20\%$, $30\%$, 
and $100\%$.}
\end{figure}

The effect of approximating the folding arc by a polygonal, piecewise straight
curve is illustrated in Figure~\ref{fig:pw_linear_arc}. The plots display 
the deformations obtained for the circular arc approximated accurately
with piecewise quadratic edges of elements to a coarse approximation using
three straight segments. The Frobenius norm of the Hessian, i.e., an approximation
of the mean curvature of the deformed plates is visualized via a gray scale
coloring. We see that energy concentrations occur at the kinks of the piecewise
linear arc while a more uniform distribution arises for the circular arc with
moderate peaks at the boundary where the compressive boundary condition is
imposed and where the arc ends. Apart from that the overall deformation does
not differ significantly and the main difference is a less curved plate away
from the arc for the coarse, piecewise linear approximation.  

\begin{figure}[htb]
\includegraphics[width=0.43\textwidth]{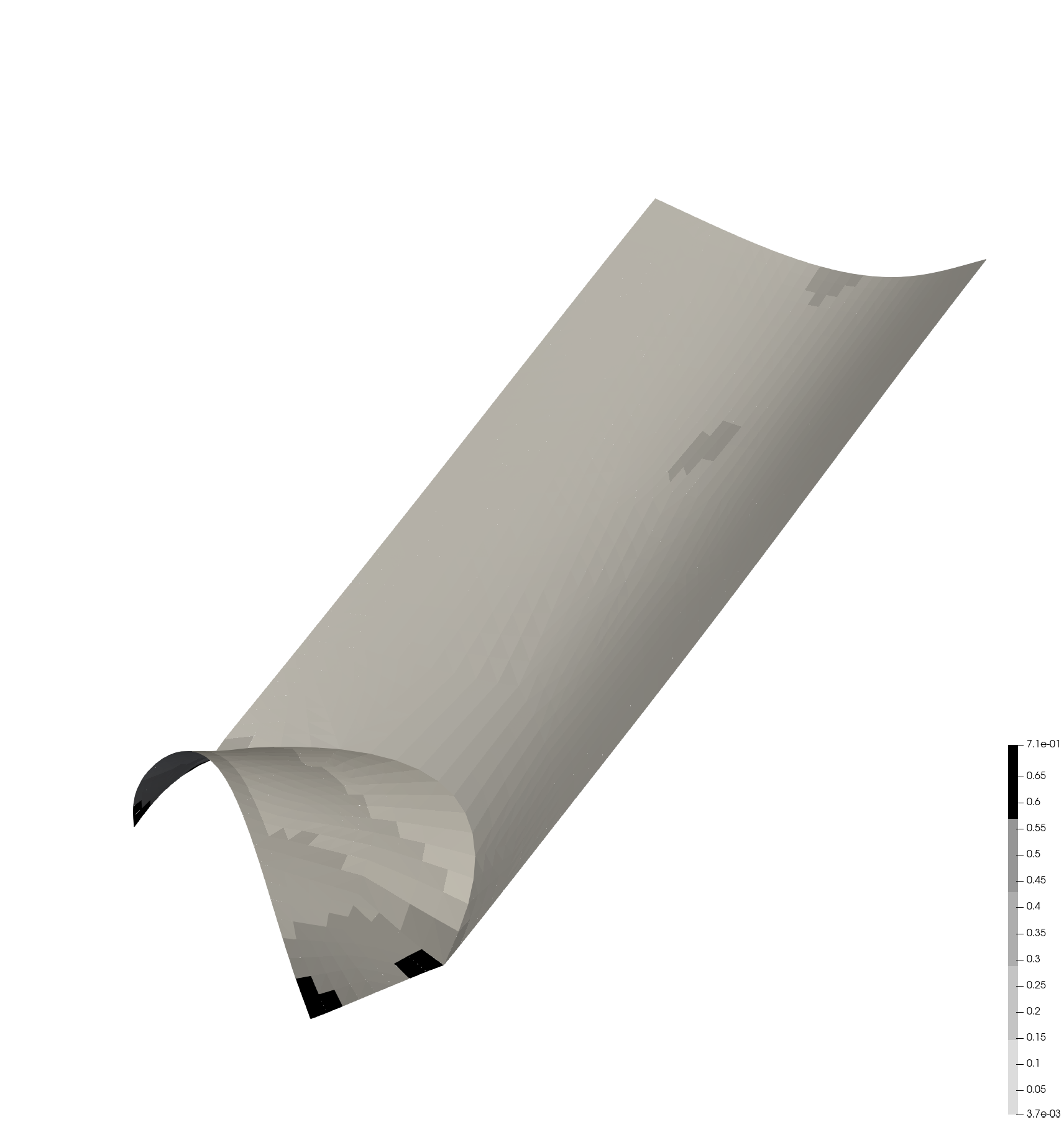} \hspace*{-2mm} 
\includegraphics[width=0.43\textwidth]{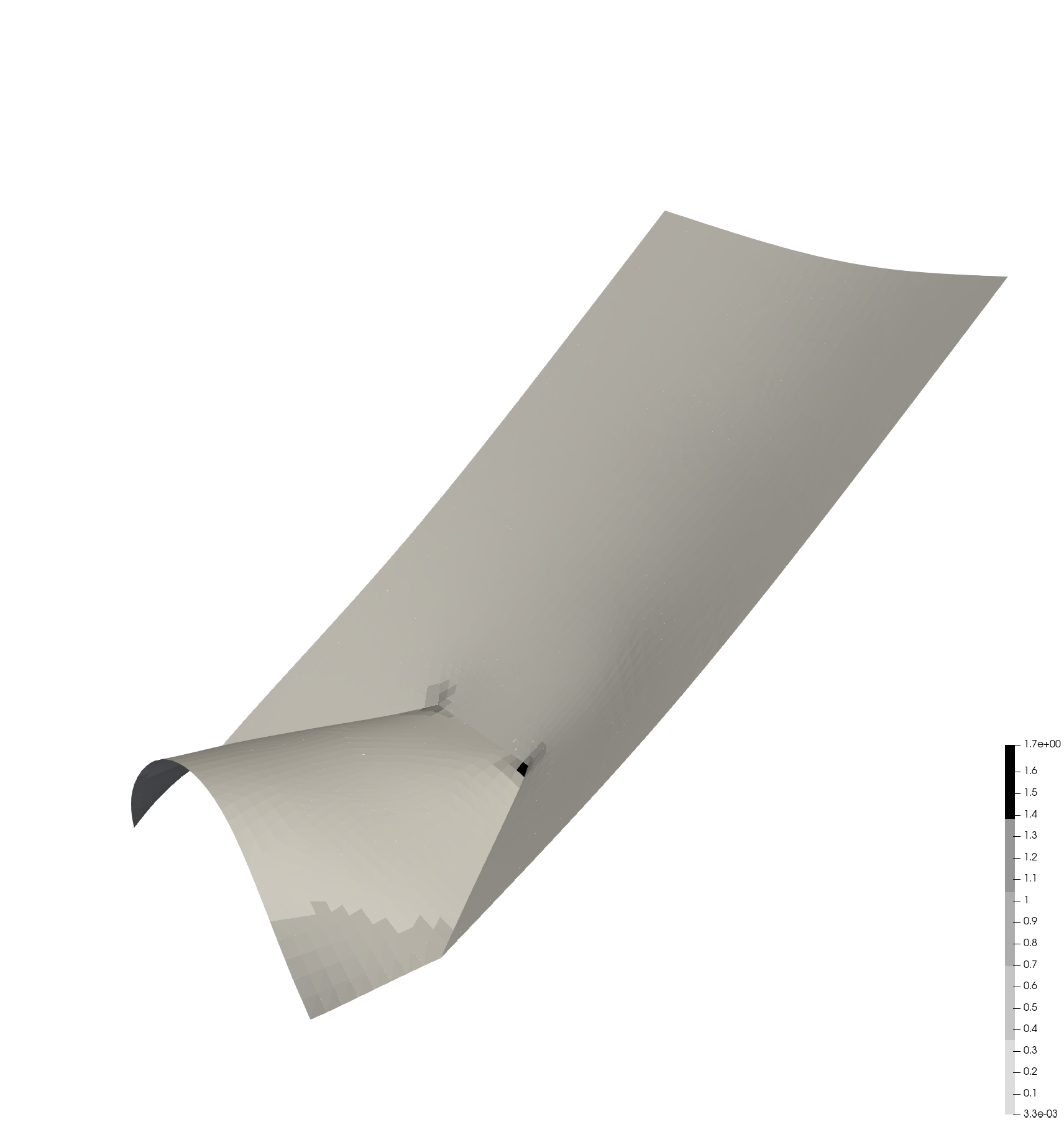} 
\caption{\label{fig:pw_linear_arc} Different approximations of a circular
folding arc using a piecewise quadratic segments (left) and piecewise 
linear segments (right) lead to different energy contributions but 
similar deformations. The colors represent the average curvature over each element of the subdivision. The ranges are from 0 (white) to 0.71 (black) for the quadratic folding line and from 0 (white) to 1.7 (black) for the piecewise linear folding line.}
\end{figure}

\subsection{Paper cutting and bending}
Our second experiment simulates a typical origami folding construction 
with curved arcs which is also known as kirigami folding which includes cutting
and bending a piece of paper.
In our example a square domain with a square hole is prepared using four arcs
that connect midpoints of the outer boundary with the corners of the 
inner boundary, cf.~Figure~\ref{fig:ex_flower_setup}. The precise
settings are as follows.

\begin{example}[Flower configuration]\label{ex:flower}
Let $S = (0,16) \times (0,16) \setminus S'$, where $S'$ is the square with 
defined by the vertices $x_1= (6,7)$, $x_2=(7, 10)$, $x_3=(10,9)$, and $x_4=(9,6)$. 
We use four cubic Bezier curves that connect the midpoints $x_{M,1} = (8,0)$, 
$x_{M,2} = (16,8)$, $x_{M,3} = (8,16)$, and $x_{M,4}= (0,8)$ of the outer sides of
$S$ with the points $x_3,x_2,x_1,x_4$, respsectively, using suitable control
points, e.g., for the first arc $\Sigma^1$ we use 
\[\begin{split}
\hx_{1,1} &= (10+3\cos(\alpha)-\sin(\alpha),9-\cos(\alpha)-3\sin(\alpha)), \\
\hx_{1,2} &= (8+3.162\cos(\alpha),3.162\sin(\alpha)),
\end{split}\]
with $\a=\pi/6$. Control points for the arcs $\Sigma^\ell$, $\ell=2,3,4$, are
obtained via rotational point symmetry, cf. Figure~\ref{fig:ex_flower_setup}. Compressive
boundary conditions with $s=60\%$ compression rate are imposed at the opposite 
boundary points $x_{M,1}$ and $x_{M,3}$.
\end{example}

The setting and a photo of the result of a real experiment corresponding
to Example~\ref{ex:flower} are shown in Figure~\ref{fig:ex_flower_setup}.

\begin{figure}[htb]
\includegraphics[width=0.9\textwidth]{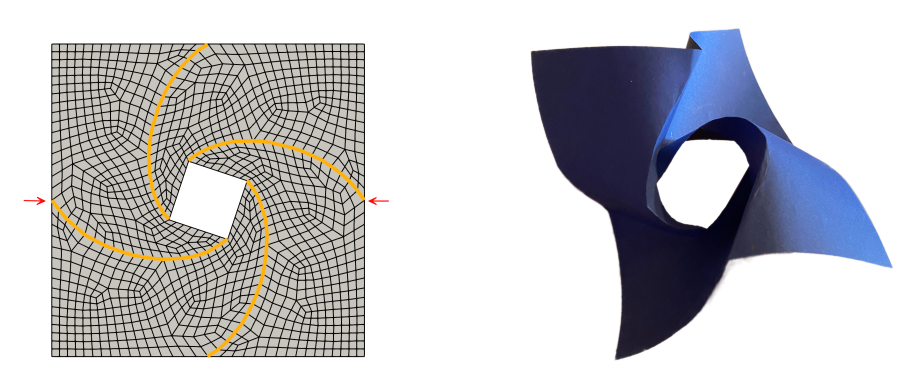}
\caption{\label{fig:ex_flower_setup} Geometric setting of Example~\ref{ex:flower} 
(left) and result of a real experiment with $s=60\%$ compression rate (right).}
\end{figure}

Numerical solutions for Example~\ref{ex:flower} for different compression 
rates are shown in Figure~\ref{fig:ex_flower_stationary}. We used a triangulation
with~1904 elements, a pseudo-time step $\tau=0.025$,
termination tolerances $\veps_{stop} = 0.3$, $\veps_{pp} =0.5$. The 
discrete, nearly isometric deformations obtained with our numerical
scheme reveal a remarkable similarity to configurations obtained in real 
experiments. 

\begin{figure}[htb]
\includegraphics[width=0.45\textwidth]{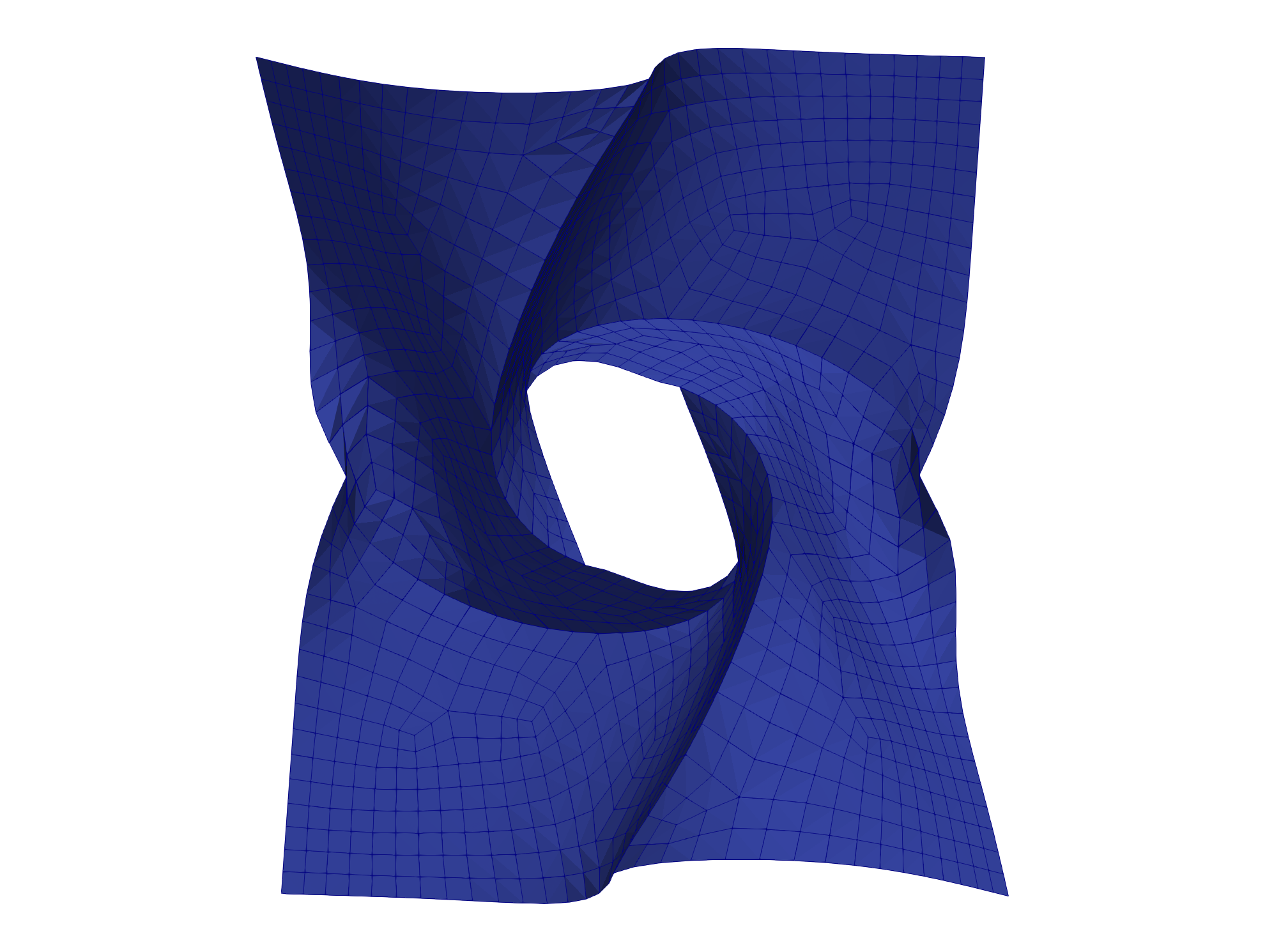} \vspace*{-3mm}
\includegraphics[width=0.45\textwidth]{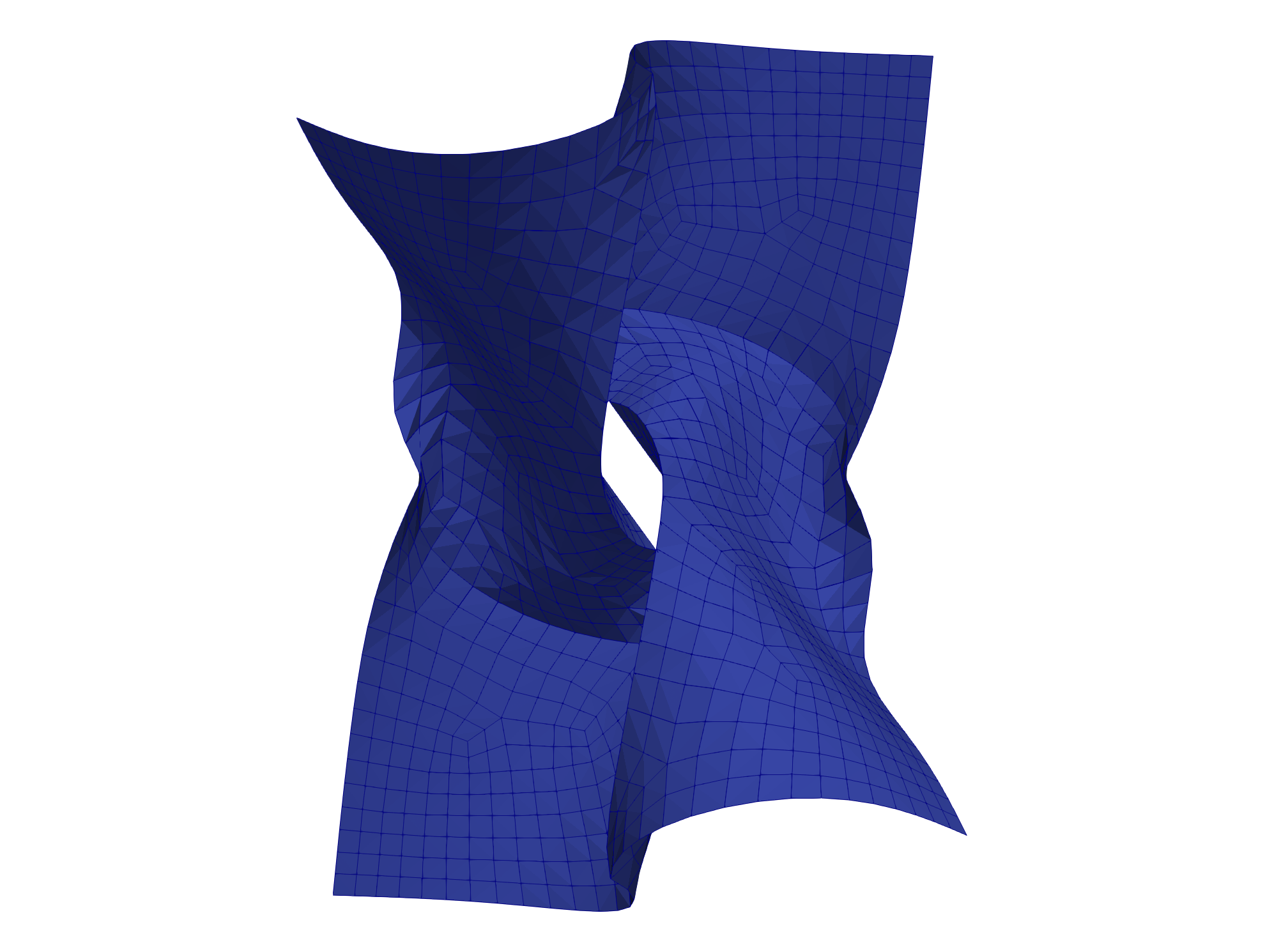}\\[3mm]
\includegraphics[width=0.45\textwidth]{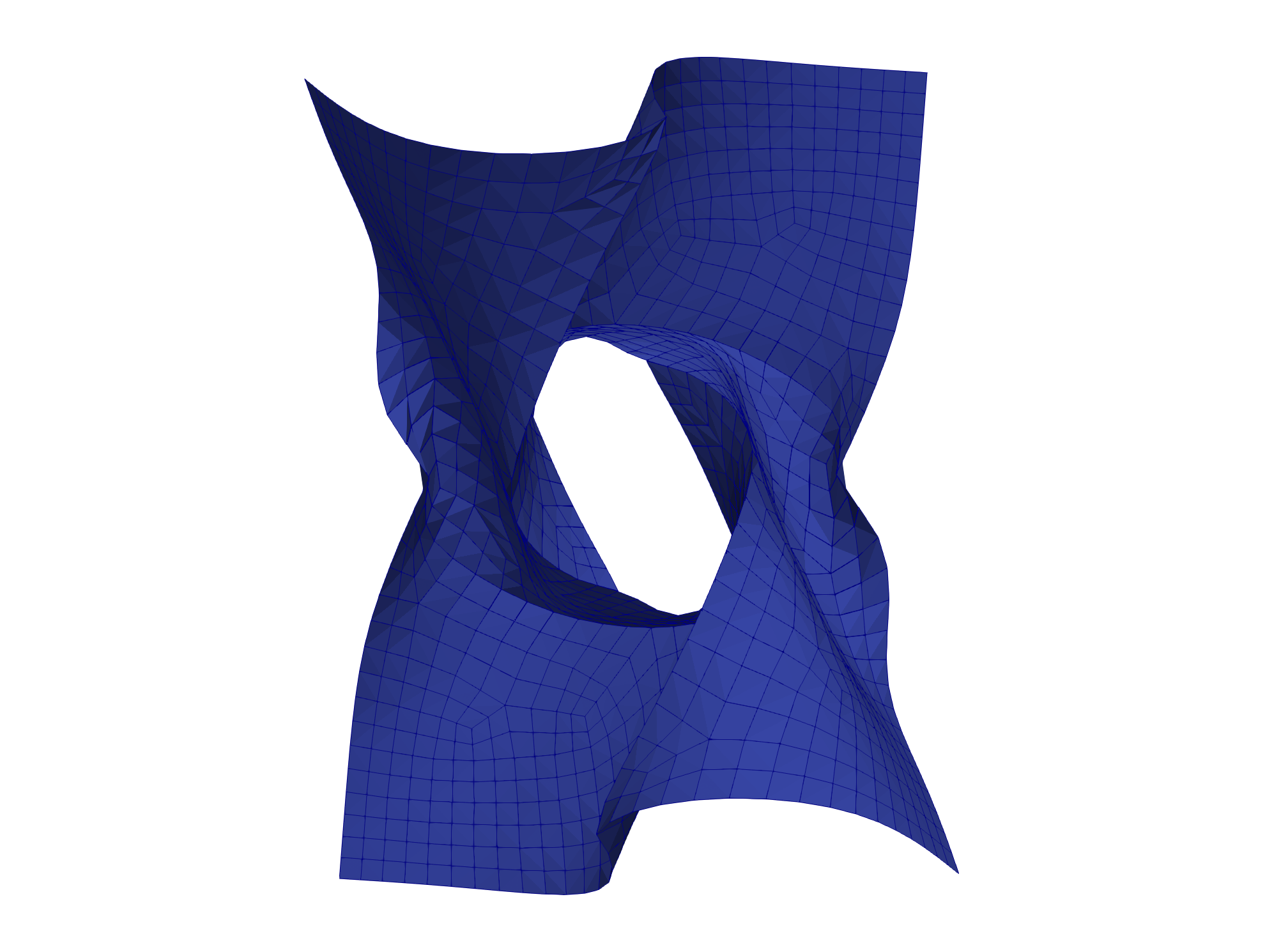}\vspace*{-3mm}
\includegraphics[width=0.45\textwidth]{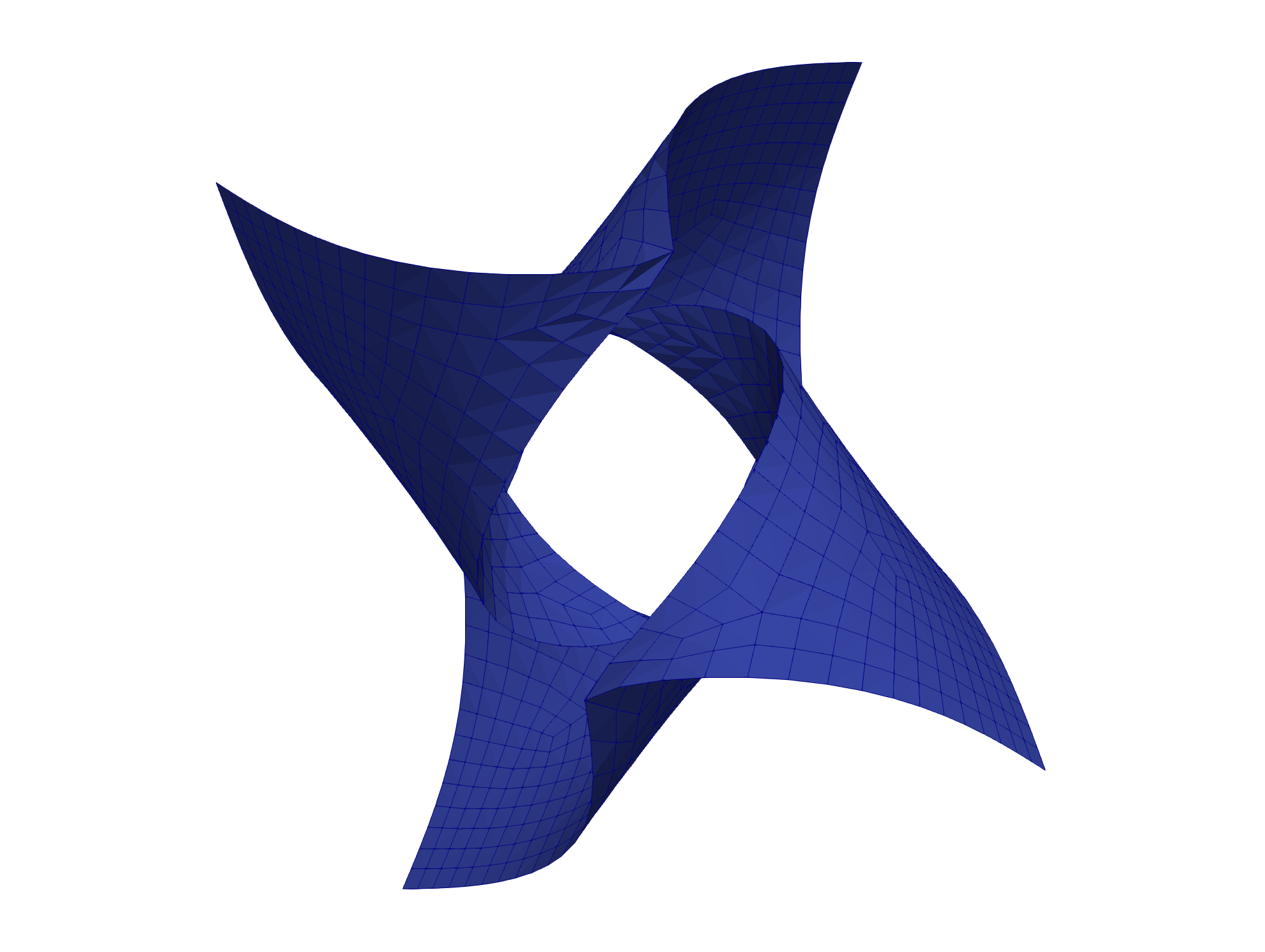} 
\caption{\label{fig:ex_flower_stationary} 
Discrete, nearly stationary deformations for $37.5\%$, $55\%$, $57.5\%$, and $60\%$ 
compression rates (left to right, top to bottom).}
\end{figure}

\subsection*{Acknowledgments} 
The authors SB and PH acknowledge support by the DFG via the priority programme
SPP 2256 {\em Variational Methods for Predicting Complex Phenomena in Engineering 
Structures and Materials} (BA 2268/7-1, HO 4697/2-1).
The author AB is partially supported by NSF grant DMS-2110811.

\section*{References}
\printbibliography[heading=none]

\end{document}